\documentclass[11pt,reqno]{amsart}
\usepackage{amssymb,amsmath,amsthm}
\usepackage{comment,enumitem,url}
\usepackage{xcolor}
\usepackage{cleveref}
\usepackage{mathtools}
\usepackage{a4wide,fullpage}
\usepackage{bm}
\usepackage{graphicx,amsfonts}

\usepackage[normalem]{ulem}

\usepackage{url}

\usepackage{color}

\newcommand{\leg}[2]{\left(\frac{#1}{#2}\right)}

\newcommand{\N}{\mathbb{N}}
\newcommand{\R}{\mathbb{R}}

\newcommand{\s}{{\sigma}}

\renewcommand{\a}{\alpha}
\renewcommand{\b}{\beta}

\renewcommand{\d}{{\delta}}
\newcommand{\g}{\gamma}

\renewcommand{\l}{\lambda}

\newcommand{\z}{\zeta}

\renewcommand{\(}{\left\(}
\renewcommand{\)}{\right\)}

\newcommand{\pa}[2]{\left(\frac{#1}{#2}\right)}

\newcommand{\od}{\mathrm{od}}
\numberwithin{equation}{section}
\theoremstyle{plain}
\newtheorem{theorem}{Theorem}[section]
\newtheorem{lemma}[theorem]{Lemma}

\newtheorem{remark}[theorem]{Remark}
\newtheorem*{remark*}{Remark}

\newtheorem{corollary}[theorem]{Corollary}
\newtheorem{proposition}[theorem]{Proposition}

\theoremstyle{definition}

\numberwithin{equation}{section}

\setlist[enumerate]{leftmargin=*,label=\rm{(\arabic*)}}

\makeatletter
\@namedef{subjclassname@2020}{%
	\textup{2020} Mathematics Subject Classification}
\makeatother

\allowdisplaybreaks

\makeatletter
\newcommand{\vast}{\bBigg@{4}}
\newcommand{\Vast}{\bBigg@{5}}
\makeatother

\title{Restricted Overpartitions and concave compositions: their modularity and asymptotics}
\author{Koustav Banerjee}
\author{Kathrin Bringmann}
\author{Atul Dixit}
\address{University of Cologne, Department of Mathematics and Computer Science, Weyertal 86-90, 50931 Cologne, Germany}
\email{kbanerj1@uni-koeln.de}
\email{kbringma@uni-koeln.de}

\address{Department of Mathematics, Indian Institute of Technology Gandhinagar, Palaj, Gandhinagar 382355, Gujarat, India}
\email{adixit@iitgn.ac.in}

\makeatletter
\@namedef{subjclassname@2020}{%
	\textup{2020} Mathematics Subject Classification}
\makeatother

\subjclass[2020]{11F03, 11F37, 11P82, 11P83, 33D15.}
\keywords{ Asymptotics, concave compositions, false theta functions, mock theta functions, mock Maass theta functions, modular forms, overpartitions, Tauberian theorem.}

\begin{document}

\begin{abstract}
In this paper we study restricted overpartitions and concave compositions. In several cases the resulting generating functions involve simultaneously modular forms, mock theta functions, mock Maass theta functions, and false theta functions, illustrating the appearance of mixed modular structures in restricted partition problems. Moreover, we obtain their asymptotic main terms. We also study related rank statistics.

\end{abstract}

\maketitle



\section{Introduction and Statements of results}   

A {\it{partition}} of $n\in \N_0$ is a non-increasing sequence of positive integers $\l_1,\l_2,\dots,\l_r$ such that $ \sum_{j=1}^{r}\l_j=n$. The partition function $p(n)$ counts the number of partitions of $n$. Its generating function is given by 
\begin{equation*}
P(q):=\sum_{n\ge 0}p(n)q^n=\frac{1}{\left(q\right)_{\infty}}. 
\end{equation*}
Here, and throughout, $|q|<1$, and for $a\in\mathbb{C}$ and $n\in\mathbb{N}_0\cup \{\infty\}$, we set $(a)_n\hspace{-0.1 cm}=\hspace{-0.1 cm}(a;q)_n\hspace{-0.1 cm}:=\prod_{j=0}^{n-1}(1-aq^j)$.
Note that $P(q)=\frac{q^{\frac{1}{24}}}{\eta(\tau)}$, where $\eta(\tau)$ is the {\it Dedekind $\eta$-function} defined by
\begin{equation*}
\eta(\tau) :=q^{\frac{1}{24}}\left(q\right)_{\infty}\ \left(q:=e^{2\pi i \tau}, \tau\in\mathbb{H}\right).
\end{equation*}
As $\eta$ is a modular form of weight $\frac 12$ for SL$_2(\mathbb{Z})$ (with multiplier), $P$ is a weakly holomorphic modular form. Hardy and Ramanujan \cite{HR} developed the Circle Method, to obtain an asymptotic expansion of $p(n)$ and in the simplest form, its main term equals
\begin{equation}\label{HRasymptotic}
p(n)\sim \frac{e^{\pi\sqrt{\frac{2n}{3}}}}{4\sqrt{3}n}\ \ \text{as}\ \ n \rightarrow \infty.
\end{equation}
Dyson \cite{Dyson} introduced the so-called rank of a partition to give a combinatorial explanation of certain congruences of Ramanujan (see  \cite{Ramanujan}) for $p(n)$. 
The {\it rank} of a partition is defined as the largest part minus the number of its parts. For $r\in \mathbb{N}$ and $a_1,\dots,a_r\in \mathbb{C}$, we let $(a_1,\dots,a_r)_n=(a_1,\dots,a_r;q)_n=(a_1)_n\dots (a_r)_n$. The generating function of $N(m,n)$, the number of partitions of $n$ with rank $m$, is given by (see \cite{AS})
\begin{equation}\label{rank1}
R(\z;q):=\sum_{\substack{n\ge 0\\m\in\mathbb{Z}}}N(m,n)\z^mq^n=\sum_{n\ge 0}\frac{q^{n^2}}{\left(\z q,\z^{-1}q\right)_n}.
\end{equation}
Note that $f(q)=R(-1;q)$ is one of the mock theta functions introduced by Ramanujan in his last letter to Hardy. For Ramanujan \cite{BB}, a {\it mock theta function} is a function $F$ that has asymptotic expansions at every roots of unity\footnote{Equivalently, considering $F$ as a function of $\tau$, we can also consider asymptotic expansions at rationals.} of the same type as those of weakly holomorphic modular forms, however there is not a single weakly holomorphic modular form whose asymptotic expansions agree at all roots of unity with that of $F$. In his thesis \cite{Zw}, Zwegers ``completed" Ramanujan's mock theta functions to  harmonic Maass forms\footnote{A {\it harmonic Maass form} is a function that satisfies modularity, is annihilated by the weight $k$ hyperbolic Laplacian, and satisfies a growth condition at the cusps. Its holomorphic part is called a {\it mock modular form}.} by adding additional non-holomorphic terms. Ono and the second author \cite{BO} showed that, for roots of unity $\z\neq 1$, $R(\zeta;q)$ is the holomorphic part of a weight $\frac 12$ harmonic Maass form. Later Berkovich and Garvan \cite[Section 5]{BG} defined the $M_2$-rank for partitions without repeated odd parts as the smallest integer at least $\frac{l(\l)}{2}$ minus the number of parts of $\l$, where $l(\l)$ is the largest part of $\l$. The $M_2$-rank generating function for partitions without repeated odd parts is given by
\begin{equation}\label{rank2}
\operatorname{R2}(\z;q):=\sum_{\substack{n\ge 0\\m\in\mathbb{Z}}}N_2(m,n)\z^mq^n=\sum_{n\ge 0}\frac{\left(-q;q^2\right)_nq^{n^2}}{\left(\z q^2,\z^{-1}q^2;q^2\right)_n},
\end{equation}
where $N_2(m,n)$ counts the number of partitions of $n$ without repeated odd parts with $M_2$-rank $m$  (see \cite[equation (1.1)]{LO}). Similar as for $R(\zeta;q)$, Ono, Rhoades, and the second author \cite{BOR} showed that if $\z$ is a root of unity, then $\operatorname{R2}(\z;q)$ is a mock modular form. A {\it mixed mock modular form} is a linear combination of product of modular forms and mock theta functions. An important example comes from concave compositions. A {\it concave composition} of $n$ is a sum of integers of the form
$\sum_{j=1}^{r}a_j+\underline{c}+\sum_{j=1}^{s}b_j=n\ \ \underline{c},r,s\in\mathbb{N}_0$,
where $a_1\ge a_2\ge\dots\ge a_r> \underline{c}< b_s\le\dots\le b_2\le b_1$. The distinguished point $\underline{c}$ is the {\it central part} of the sequence. Let $v(n)$ count the number of concave compositions of $n$. Its generating function is given by (see \cite{ARZ}) 
\begin{equation*}
V(q):=\sum_{n\geq0}v(n)q^n=\sum_{n\geq0}\frac{q^n}{(q^{n+1})^2_{\infty}}.
\end{equation*}
By \cite{ARZ}, $V$ is a mixed mock modular form. We also require the theta function
\begin{equation}\label{Ramtheta}
\Theta(-q):=\sum_{n\in \mathbb{Z}}(-1)^n q^{n^2}=\frac{\left(q\right)^2_{\infty}}{\left(q^2;q^2\right)_{\infty}}
\end{equation}
which is a modular form of weight $\frac 12$. The Fourier coefficients of its reciprocal count overpartitions \cite{CL}. An {\it overpartition} of $n$ is a partition of $n$ in which the first occurrence of a part may be overlined. Let $\overline{p}(n)$ denote the number of overpartitions of $n$. 
We denote the generating function of $\overline{p}(n)$ by $\overline{P}$. 

 {\it False theta functions} are theta functions but twisted with sign-factors. Due to this ``twist", they are not modular. 
 Now, we give an example of a sequence whose generating function involves a false theta function. A {\it unimodal sequence} of $n$ is a finite sequence of positive integers $a_j~ (1\le j\le r)$, $b_{\ell}~ (1\le \ell\le s)$, $r,s\in \mathbb{N}_0$, and $\underline{c}\in \mathbb{N}$ such that $a_r \leq \ldots \le a_2\leq a_1\le \underline{c} \geq b_1 \geq b_2\ge \ldots \geq b_s$ for some $k\in \mathbb{N}$ and $\sum_{j=1}^{r}a_j+\underline{c}+\sum_{j=1}^{s}b_j=n$. The generating function of\footnote{In \cite[Section 3]{And}, $u(n)$ was denoted by $\s\s(n)$ and was called stacks with summits of size $n$.} $u(n)$, the number of unimodal sequences of size $n$, is given by (see \cite{And})
\begin{align*}
U(q):=\ \sum_{n\ge 0}u(n)q^n=\sum_{n\ge 0}\frac{q^n}{(q)^2_n}= P^2(q)\sum_{n\geq0} (-1)^n q^{\frac{n(n+1)}{2}}.
\end{align*}
Thus $U$ is a {\it mixed false theta function}, which a linear combinations of false theta functions multiplied by modular forms. Like Jacobi theta functions fit into the theory of Jacobi form, by introducing an extra Jacobi variable, Nazaroglu and the second author \cite[Theorem 1.1]{BN} showed that the ``modular completion" of a two-variable false theta functions transforms like a Jacobi form.

Next we turn to rank statistics of restricted partitions. Andrews \cite{AE} considered partitions into distinct parts with even rank minus the number of those with odd rank. He showed that the corresponding generating function is Ramanujan's $\s$-function 
\begin{equation}\label{sigma function}
\sigma(q):=\sum_{n\geq0}\frac{q^{\frac{n(n+1)}{2}}}{(-q)_n}.
\end{equation}
Andrews, Dyson, and Hickerson \cite{ADH} proved that the Fourier coefficients of $\s$ assume all integers infinitely often by relating them to the arithmetic of $\mathbb{Q}(\sqrt{6})$. Cohen \cite{Cohen} then constructed a Maass form from the Fourier coefficients of a linear combination of $\sigma(q)$ and a ``companion". These two functions are related to quantum modular forms \cite{Zagier}. A function $f:\mathbb{Q}\setminus S\rightarrow\mathbb{R}$ (for some discrete subset $S$ of $\mathbb{Q}$) is a {\it quantum modular form} if, for $\g=\begin{psmallmatrix}
a& b\\
c& d
\end{psmallmatrix}\in \text{SL}_2(\mathbb{Z})$, the {\it obstruction to modularity} $f_{\g}(x):=f(x)-(cx+d)^{-k}f(\frac{ax+b}{cx+d})$ is ``nice"\footnote{In typical examples, $f_{\g}$ extends to a real-analytic function on $\mathbb{R}$ except for a finite set of points.}. Zwegers \cite{Zwegers} developed a broader framework by constructing {\it mock Maass theta functions} which are certain theta functions that are non-modular eigenfunctions of the hyperbolic Laplacian which can be completed to non-holomorphic modular functions. 

 Various examples of modular and mock modular generating functions arising from partition statistics have been studied extensively in the literature. Classical instances include generating functions of ranks, unimodal sequences, and concave compositions. Here the corresponding generating functions are closely related to modular forms or mock theta functions. In contrast, the restricted overpartition generating functions studied in this paper lead naturally to decompositions involving several different modular objects simultaneously. This highlights a new setting in which mixed modular structures appear in a combinatorial context.



 Motivated by the rich modular behaviour exhibited by generating functions of restricted partitions and overpartitions, we investigate several families of overpartitions with parity and overlining restrictions. Such restrictions often lead to generating functions that admit tractable $q$-series representations while still exhibiting interesting modular properties. In this paper, we consider overpartitions of $n$ with the following restrictions:
\begin{enumerate}
	\item The smallest part and the even parts are non-overlined,
	\item If the smallest part is odd, then it appears exactly once.
\end{enumerate}
Let $\overline{p}_{\text{od}}(n)$ denote 
the number of such overpartitions of $n$.  
For example, $\overline{p}_{\text{od}}(6)=7$, since the admissible overpartitions are $6,\ 5+1,\ \overline{5}+1,\ 4+2,\ 3+2+1,\ \overline{3}+2+1,\ 2+2+2.$
We have 
\begin{equation}\label{pod}
\overline{P}_{\rm od}(q):=\sum_{n\ge 0}\overline{p}_{\text{od}}(n)q^n=\sum_{n\ge 0}\frac{\left(-q^{2n+3};q^2\right)_{\infty}q^{2n+1}}{\left(q^{2n+2},q^{2n+3};q^2\right)_{\infty}}+\sum_{n\ge 1}\frac{\left(-q^{2n+1};q^2\right)_{\infty}q^{2n}}{\left(q^{2n},q^{2n+1};q^2\right)_{\infty}}.
\end{equation}
Here we show that $\overline{P}_{\textup{od}}$ can be written in terms of a false theta function and a mixed mock Maass theta function.\footnote{{\it Mixed mock Maass theta functions} are linear combinations of products of modular forms and mock Maass theta functions.} For this define,\footnote{This is a mock Maass theta function by work of Zwegers (see the discussion at the end of Section 6 of \cite{Zwegers}). Similar to the case of $\s$, the Fourier coefficients of $W_1$ are related to the arithmetic of $\mathbb{Q}(\sqrt{2})$.} following Corson, Favero, Leisinger, and Zubairy \cite{CFLZ},
\begin{equation}\label{corson}
	W_1(q):=\sum_{n\ge 0}\frac{(-1)^n\left(q\right)_{n} q^{\frac{n(n+1)}{2}}}{\left(-q\right)_{n}}.
\end{equation} 

A notable feature of the generating functions arising in this work is that they involve simultaneously several types of modular objects. While generating functions in partition theory are often related to a single modular object (e.g. a (mock) modular form), the examples studied here naturally lead to decompositions involving several different modular structures at once. For instance, in \Cref{thm1} we express the generating function Pod(q) as a sum of a mixed mock Maass theta function, a modular form, and a false theta function. Similar phenomena appear for the other generating functions considered in this paper. These examples suggest that natural restrictions on overpartitions can lead to generating functions exhibiting mixed modular behaviour.

\begin{theorem}\label{thm1}
	We have
	\begin{align*}
	\overline{P}_{\textup{od}}(q)\!&=\!-\frac{(1+q)\left(-q;q^2\right)_{\infty}}{q\left(q\right)_{\infty}}W_1(q)\!+\!\frac{(1+q)\left(-q;q^2\right)_{\infty}}{q\left(q\right)_{\infty}}-\!\frac{1+q}{q}\sum_{n\in \mathbb{Z}}\textnormal{sgn}(n)q^{\frac{n(3n-1)}{2}}.
	\end{align*}
\end{theorem}
\noindent Using \Cref{thm1} we obtain the following asymptotic main term for $\overline{p}_{\text{od}}(n)$.
\begin{corollary}\label{thm1a}
	We have, as $n\to\infty$, 
	\begin{equation*}
	\overline{p}_{\od}(n) \sim \frac{5\pi}{48\sqrt{2}n^{\frac{3}{2}}} e^{\pi\sqrt{\frac{5n}{6}}}.
	\end{equation*}
\end{corollary}


We next consider overpartitions without repeated odd parts with the additional restrictions:
\begin{enumerate}
	\item The odd parts are non-overlined.
	\item  If the smallest part is even and occurs only once, then it is also non-overlined. 
\end{enumerate}
 Let $\overline{p}_{\textup{ev}}(n)$ denote 
 the number of such weighted overpartitions of $n$. 
We have $\overline{p}_{\textup{ev}}(6)=8$, since the admissible overpartitions are  
$6,\ 5+1,\ 4+2,\ \overline{4}+2,\ 3+2+1,\ 3+\overline{2}+1,\ 2+2+2,\ \overline{2}+2+2$. 
It is not difficult to see that the generating function of $\overline{p}_{\textup{ev}}(n)$ is given by 
\begin{align}\label{thm2eqn0}
\overline{P}_{\text{ev}}(q)&:=\sum_{n\ge 1}\overline{p}_{\textup{ev}}(n)q^n=\sum_{n\ge 0}\frac{\!\left(-q^{2n+2},-q^{2n+3};q^2\right)_{\infty}q^{2n+1}}{(q^{2n+2};q^2)_{\infty}}+\sum_{n\ge 1}\frac{\left(-q^{2n},-q^{2n+1};q^2\right)_{\infty}q^{2n}}{(q^{2n};q^2)_{\infty}}.
\end{align}
The following theorem shows that $\overline{P}_{\textup{ev}}$ is a mixed mock modular form. To be more precise, consider Ramanujan's third order mock theta function 
\begin{equation}\label{Ramanujanphi}
\phi(q):=\sum_{n\ge0}\frac{q^{n^2}}{\left(-q^2;q^2\right)_n}.
\end{equation}
\begin{theorem}\label{thm2}
	We have 
	\[
	\overline{P}_{\textup{ev}}(q)=\frac{1+q}{2(q)_{\infty}}\left(1-\frac{(q)_{\infty}^2}{(-q)_{\infty}^{2}}\right)+(1+q)\left(\phi(-q)-1\right).
	\]
\end{theorem}
\noindent As an application of \Cref{thm2}, we determine the asymptotic main term of $\overline{p}_{\textup{ev}}(n)$. 
\begin{corollary}\label{thm6.2}
	We have, as $n\to \infty$,
	\begin{equation*}
	\overline{p}_{\textup{ev}}(n)\sim \frac{e^{\pi \sqrt{\frac{2n}{3}}}}{4\sqrt{3}n}.
	\end{equation*}
\end{corollary}
As an application, we obtain asymptotics of $g(n)$, where $g(n)$ counts certain restricted colored partitions studied by Andrews and Kumar \cite{andrews-kumar} whose generating function is given by 
\begin{align*}
\sum_{n\geq0} g(n)q^n:=\sum_{n\ge0}\frac{q^{2n+1}}{(q^{2n+1})_{\infty}(q^{2n+2};q^2)_n}.
\end{align*}
 Using this, \cite[Theorem 6.1]{andrews-kumar}, Theorem \ref{thm2}, and the facts that $g(0)=0$ and $g(1)=1$, we obtain
\begin{align}\label{reln}
g(n)+\overline{p}_{\textup{ev}}(n-1)=2p(n-1).
\end{align}
Employing this, \Cref{thm6.2}, and \eqref{HRasymptotic}, we have the following.
\begin{corollary}\label{gcornew}
We have, as $n\to \infty$,
\[
g(n) \sim \frac{e^{\pi \sqrt{\frac{2n}{3}}}}{4\sqrt{3}n}.\]
\end{corollary}

We also give a refinement of $\overline{p}_{\text{od}}(n)$. Let $\overline{p}_{\text{od}}(m,n)$ count those overpartitions which are also counted by $\overline{p}_{\text{od}}(n)$, where the total number of even parts is $m$ counted in the following way. If the smallest part of $\l$ is odd, then $m$ counts the number of even parts and if the smallest part of $\l$ is even, then $m$ counts the number of even parts minus one. 
 For example, taking $n=6$, we have:
\vspace{0.3 cm} 
\begin{center}
	{\renewcommand{\arraystretch}{1.3}%
		\begin{tabular}{|c|c|c|c|c|c|c|c|}
			\hline
			$\lambda $ 
			& $6$ 
			& $5+1$ 
			& $\overline{5}+1$ 
			& $4+2$ 
			& $3+2+1$ 
			& $\overline{3}+2+1$ 
			& $2+2+2$ \\
			\hline
			$m$ 
			& $0$ 
			& $0$ 
			& $0$ 
			& $1$ 
			& $1$ 
			& $1$ 
			& $2$ \\
			\hline
		\end{tabular}
	}
\end{center}
\vspace{0.3 cm}
Thus $\overline{p}_{\textup{od}}(0,6)=3$, $\overline{p}_{\textup{od}}(1,6)=3$, and $\overline{p}_{\textup{od}}(2,6)=1$. We have
\begin{equation}\label{zgen2}
\overline{P}_{\textup{od}}(\z;q):=\!\sum_{m,n\ge 0}\!\overline{p}_{\textup{od}}(m,n)\z^m q^n=\sum_{n\ge 0}\frac{\left(-q^{2n+3};q^2\right)_{\infty}q^{2n+1}}{\left(\z q^{2n+2},q^{2n+3};q^2\right)_{\infty}}+\sum_{n\ge 1}\frac{\left(-q^{2n+1};q^2\right)_{\infty}q^{2n}}{\left(\z q^{2n},q^{2n+1};q^2\right)_{\infty}}.
\end{equation}
In \Cref{Thm1}, we give a new representation of $\overline{P}_{\textup{od}}(\z;q)$ such that for special values of $\z$, we see interesting modular properties. To state our result, we define an extension of the third order mock theta function\footnote{Note that we consider $f(\z_1,\z_2;q)$ as a meromorphic function which may have poles. These poles turn out to be important in the proof of \Cref{newthm2}.} (see \cite{Choi}) 
\begin{align}\label{Choi1}
f\left(\zeta_1,\zeta_2;q\right)&:=\sum_{n\ge 0}\frac{\zeta_1^n \zeta_2^{2n}q^{n^2-3n}}{\left(-\zeta_2,-\frac{\zeta_1 \zeta_2}{q}\right)_n}.
\end{align}
 From \cite{Choi}, for $k,\ell\in\mathbb{N}$, $f(q^k,q^{\ell};q)$ is a mock theta function in the sense of Ramanujan. 

\begin{theorem}\label{Thm1}
	We have
\begin{align*}
\overline{P}_{\textup{od}}(\z;q)\!=\frac{\left(1+q\right)\left(-q^3;q^2\right)_{\infty}}{\zeta q\left(\z q^2,q^3;q^2\right)_{\infty}}\sum_{n\ge 1}\!\frac{\left(-\frac 1q;q^2\right)_n\!\z^nq^{n(n+1)}}{\left(-\z q,-q^2;q^2\right)_n}-\frac{1+q}{\z q}f\left(\z q,q^2;q^2\right)+\frac{1+q}{\z q}.
\end{align*}	
\end{theorem}



For the special case $\z=q$ we 
recall Ramanujan's third order mock theta function 
\begin{equation}\label{Ramanujanf}
f(q):=\sum_{n\ge0}\frac{q^{n^2}}{\left(-q\right)^2_{n}}
\end{equation}
and a second order mock theta function by McIntosh \cite{McIntosh} 
\begin{equation}\label{McIntosh}
\mu(q):=\sum_{n\ge 0}\frac{(-1)^n \left(q;q^2\right)_n q^{n^2}}{\left(-q^2;q^2\right)^2_n}.
\end{equation}
 The following corollary shows that $\overline{P}_{\textup{od}}(q;q)$ is a mixed mock modular form.
\begin{corollary}\label{Thm1cor}
We have 
\[
\overline{P}_{\textup{od}}(q;q)=\frac{(1+q)\left(-q;q^2\right)_{\infty}}{q^2\left(q;q^2\right)^2_{\infty}}\left(-1-q+\mu(-q)\right)-\frac{1+q}{q^2}\left(f\left(q^2\right)-1\right).\]
\end{corollary}
 Next, we let $\z\to q^{-2}$ in \Cref{Thm1}. To state the result, we consider overpartitions of $n$ with the following additional restrictions:
\begin{enumerate}
	\item The smallest part is even and even parts are non-overlined.
	\item The difference between odd parts (if they appear) and the smallest even part is at least three.
\end{enumerate} 
Let $\overline{p}^{[1]}_{\text{od}}(n)$ denote the number of such overpartitions of $n$. For example, $\overline{p}^{[1]}_{\text{od}}(9)=4$, since the admissible partitions are $
7+2,\ \overline{7}+2,\ 5+2+2,\ \overline{5}+2+2$.
We define
\begin{equation}\label{genfuncnew}
\overline{P}^{[1]}_{\rm od}(q):=\sum_{n\ge 2}\overline{p}^{[1]}_{\text{od}}(n)q^n=\sum_{n\ge1}\frac{\left(-q^{2n+3};q^2\right)_\infty q^{2n}}{\left(q^{2n},q^{2n+3};q^2\right)_\infty}.
\end{equation}
Then $\overline{P}^{[1]}_{\rm od}$ is a linear combination of a mixed mock Maass theta function and a false theta function. 

\begin{theorem}\label{newthm1}
	We have
	\[
	\overline{P}^{[1]}_{\rm od}(q)=\frac{q\left(-q^3;q^2\right)_{\infty}}{\left(q\right)_{\infty}}W_1(q)+\frac{\left(1-2q-q^2\right)\left(-q^3;q^2\right)_{\infty}}{\left(1+q\right)\left(q\right)_{\infty}}+\frac{q}{1+q}\sum_{n\in\mathbb{Z}}\textnormal{sgn}(n)q^{\frac{n(3n-1)}{2}}-\frac{1}{1+q}.	
	\]
\end{theorem}
We also study a two-variable generalization of $\overline{P}_{\text{ev}}(q)$, denoted by $\overline{P}_{\textup{ev}}(\z;q)$. Let $\overline{p}_{\text{ev}}(m,n)$ count the number of restricted overpartitions of $n$ enumerated by $\overline{p}_{\textup{ev}}(n)$ such that $m$ counts the odd parts plus the non-overlined even parts.
For example, taking $n=6$, we have
\vspace{0.3 cm}
\begin{center}
	{\renewcommand{\arraystretch}{1.3}%
		\begin{tabular}{|c|c|c|c|c|c|c|c|c|}
			\hline
			$\lambda$ 
			& $6$ 
			& $5+1$ 
			& $4+2$
			& $\overline{4}+2$ 
			& $3+2+1$ 
			& $3+\overline{2}+1$
			& $2+2+2$ 
			& $\overline{2}+2+2$\\
			\hline
			$m$ 
			& $1$ 
			& $2$ 
			& $2$ 
			& $1$ 
			& $3$ 
			& $2$ 
			& $3$ 
			& $2$\\
			\hline
		\end{tabular}
	}
\end{center}
\vspace{0.3 cm}
Therefore $\overline{p}_{\textup{ev}}(1,6)=2$, $\overline{p}_{\textup{ev}}(2,6)=4$, and $\overline{p}_{\textup{ev}}(3,6)=2$. We have 
\begin{align}\label{zgen3}
\overline{P}_{\textup{ev}}(\z;q)&:=\sum_{m,n\ge 0}\overline{p}_{\textup{ev}}(m,n)\z^m q^n\nonumber\\
&=\sum_{n\ge 0}\frac{\left(-q^{2n+2},-\z q^{2n+3};q^2\right)_{\infty}\z q^{2n+1}}{\left(\z q^{2n+2};q^2\right)_{\infty}}+\sum_{n\ge 1}\frac{\left(-q^{2n},-\z q^{2n+1};q^2\right)_{\infty}\z q^{2n}}{\left(\z q^{2n};q^2\right)_{\infty}}.
\end{align}
We recall another (see \cite[p. 347]{Choi}) generalized\footnote{From \cite[p. 347]{Choi}, we know that for $k,\ell\in\mathbb{N}$, $\nu(q^{k},q^{\ell};q)$ is a mock theta function in the sense of Ramanujan.} third order mock theta function
\begin{align}\label{Choi2}
\nu\left(\zeta_1,\zeta_2;q\right)&:=\sum_{n\ge 0}\frac{ \zeta_2^{2n}q^{n(n-1)}}{\left(-\frac{\zeta_1^2 \zeta_2^2}{q^3};q^2\right)_{n+1}}.
\end{align} 
\noindent We decompose $\overline{P}_{\textup{ev}}(\z;q)$ into three pieces: one comes from $\nu$, another is essentially a modular form, and the third one has modular properties for special values of $\z$.
\begin{theorem}\label{thm4}
	We have
	\begin{align*}
	\overline{P}_{\textup{ev}}(\z;q)=\frac{\z q\left(q^2,\z q^3,\frac{1}{\z q};q^2\right)_{\infty}}{\left(-q^3,-\z q^2;q^2\right)_{\infty}}&+\frac{\left(1+q\right)\left(-q^2,-\z q;q^2\right)_{\infty}}{\left(\z q^2;q^2\right)_{\infty}}\sum_{n\ge 0}\frac{\left(-\frac{1}{\z};q^2\right)_n(-1)^n q^{n}}{\left(-q^2;q^2\right)_n}\\
	&\hspace{1.3 cm}+\left(1+q\right)\left(\left(1+\z\right)\nu\left(iq,i\left(\z q\right)^{\frac 12};q\right)-1\right).
	\end{align*}
\end{theorem}
\noindent Next, recall Ramanujan's third order mock theta function 
\begin{equation}\label{ramanujannu}
\nu(q):=\sum_{n\geq0}\frac{q^{n(n+1)}}{(-q;q^2)_{n+1}}.
\end{equation}
The following corollary shows that $\overline{P}_{\textup{ev}}(-q;q)$ is a mixed mock modular form.
\begin{corollary}\label{cor7.4}
	We have
	\begin{align*}
	\overline{P}_{\textup{ev}}(-q;q)=-1-q-\frac{2\left(-q^2;q^4\right)_{\infty}\left(q^8;q^8\right)_{\infty}}{(q^6;q^4)_{\infty}}+\frac{2\left(q^4;q^4\right)_{\infty}}{(-q^3;q^2)_{\infty}}+\left(1-q^2\right)\nu(-q).
	\end{align*}
\end{corollary} 
Moreover, $\overline{P}_{\textup{ev}}(-1;q)$ is basically a modular form. 
\begin{corollary}\label{newthm2}
We have
\[
\overline{P}_{\textup{ev}}(-1;q)=\left(1+q\right)\left(\left(q;q^2\right)_{\infty}-1\right).
\]	
\end{corollary}

Letting $\z\to q^{-2}$ in \Cref{thm4}, we obtain the following identity of Ramanujan \cite{AB5}. 

\begin{corollary}\label{RamanujanIdentity}
	We have
	\[
	2\phi(-q)-f(q)=\frac{\Theta^2(-q)}{\left(q\right)_{\infty}}.
	\]
\end{corollary}

\begin{remark}
There are at least three known proofs of \Cref{RamanujanIdentity}, namely, Watson's \cite{watson final}, Fine's \cite{Fine}, and one in \cite{AB5}.\hspace{-0.11 cm} Neither of these explicitly appear in Ramanujan's work which suggests that his own method was perhaps different. Our techniques, however, mainly employ the identities of, or well-known to, Ramanujan, as can be seen in our proof of Corollary \ref{RamanujanIdentity}.
\end{remark}
 
Finally, we consider concave compositions with the following restrictions: 
 \begin{enumerate}
 	\item The central part is even and all other parts are odd.
 	\item Odd parts on the left side of the central part may be overlined whereas odd parts to the right of the central part are non-overlined.
 \end{enumerate}
 Let $\overline{v}_{\textup{od}}(n)$ denote the number of such compositions of $n$. For example, $\overline{v}_{\textup{od}}(2)=6$, enumerated by $1+1+\underline{0},\ \overline{1}+1+\underline{0},\ 1+\underline{0}+1,\ \overline{1}+\underline{0}+1,\ \underline{0}+1+1,\ \underline{2}.$
The generating function of $\overline{v}_{\textup{od}}(n)$ is given by
\[
\overline{V}_{\textup{od}}(q):=\sum_{n\ge 0}\overline{v}_{\textup{od}}(n)q^n=\sum_{n\ge 0}\frac{\left(-q^{2n+1};q^2\right)_{\infty}q^{2n}}{\left(q^{2n+1};q^2\right)^2_{\infty}}.
\]
For such a concave composition $\l$, the \emph{rank relative to the non-overlined odd parts} is defined as $\mathcal{R}(\l):=s-r$, where $s$ (resp. $r$) is the number of non-overlined odd parts to the right (resp. to the left) of the central part of the composition. Let $\overline{v}_{\textup{od}}(m,n)$ denote the number of such concave compositions of $n$ with rank with respect to non-overlined odd parts is $m$. For $n=2$, we have:
\vspace{0.3 cm}
\begin{center}
	{\renewcommand{\arraystretch}{1.4}%
		\begin{tabular}{|c|c|c|c|c|c|c|}
			\hline
			$\lambda$
			& $1+1+\underline{0}$ & $\overline{1}+1+\underline{0}$ & $1+\underline{0}+1$
			& $\overline{1}+\underline{0}+1$ & $\underline{0}+1+1$ & $\underline{2}$ \\
			\hline
			$\mathcal{R}(\lambda)=m$
			& $-2$ & $-1$ & $0$ & $1$ & $2$ & $0$ \\
			\hline
		\end{tabular}
	} 
\end{center}
\vspace{0.3 cm}
Therefore, $\overline{v}_{\textup{od}}(-2,2)=\overline{v}_{\textup{od}}(2,2)=1$, $\overline{v}_{\textup{od}}(-1,2)=\overline{v}_{\textup{od}}(1,2)=1$, and $\overline{v}_{\textup{od}}(0,2)=2$. The generating function of $\overline{v}_{\textup{od}}(m,n)$ is given by
\begin{equation}\label{zgen1}
\overline{V}_{\textup{od}}(\z;q):=\sum_{\substack{n\ge 0\\m\in\mathbb{Z}}}\overline{v}_{\textup{od}}(m,n)\z^m q^n=\sum_{n\ge 0}\frac{\left(-q^{2n+1};q^2\right)_{\infty}q^{2n}}{\left(\z q^{2n+1},\z^{-1}q^{2n+1};q^2\right)_{\infty}}.
\end{equation}
From the following theorem, we see that $\overline{V}_{\textup{od}}(\z;q)$ is related to known rank generating functions.
\begin{theorem}\label{thm3}
	We have
	\[
	q\overline{V}_{\textup{od}}(\z;q)=\frac{\left(-q;q^2\right)_{\infty}\left(\operatorname{R2}\left(-\z;q\right)-1\right)}{\left(\z q,\z^{-1}q;q^2\right)_{\infty}} -R\left(-\z;q^2\right)+1.
	\]
\end{theorem}
\begin{remark*}
Taken together, our results illustrate that natural restrictions on overpartitions and concave compositions can lead to generating functions exhibiting mixed modular structures. In particular, our examples show that combinations of modular forms, mock theta functions, mock Maass theta functions, and false theta functions arise naturally in combinatorial generating functions.	
\end{remark*} 


 The paper is organized as follows. In \Cref{sec:prelim}, first we recall $q$-series transformations required for this paper and state Ingham's Tauberian theorem and the Euler--Maclaurin summation formula. We prove \Cref{thm1} and \Cref{thm1a} in \Cref{genfn}. 
 Section \ref{sec5} presents proofs of \Cref{thm2} and \Cref{thm6.2}. 
  In \Cref{z:Gen}, we show Theorem \ref{Thm1}, Corollary \ref{Thm1cor}, Theorems \ref{newthm1} and \ref{thm4}, Corollaries \ref{cor7.4}, \ref{newthm2}, and \ref{RamanujanIdentity}, and \Cref{thm3}. Finally in \Cref{cr}, we state problems for future research. 



\section*{Acknowledgements}
The first two authors have received funding from the European Research Council (ERC)
under the European Union’s Horizon 2020 research and innovation programme (grant agreement
No. 101001179). The third author is supported by the Swarnajayanti Fellowship grant SB/SJF/2021-22/08 of ANRF (Government of India) and by the N Rama Rao chair professorship at IIT Gandhinagar. He also thanks the University of Cologne for its warm hospitality during his visit in May 2025.

\section{Preliminaries}\label{sec:prelim}

\subsection{$q$-series identities} 
 
Here we list $q$-series identities which are used in the proofs of our theorems. We begin with two identities of Ramanujan. The first was stated in \cite[equation (3.17)$_\text{R}$]{A0}. 


\begin{lemma}\label{lemident1} 
We have 
\begin{align*}
&\sum_{n\ge 0}\!\frac{\left(-aq,-bq\right)_nq^{n+1}}{\left(-cq\right)_n}\!
=\!\frac{c}{ab}\!\sum_{n\ge 1}\!\frac{\left(-c^{-1}\right)_n \pa{ab}{c}^n\! q^{\frac{n(n+1)}{2}}}{\left(\frac{aq}{c},\frac{bq}{c}\right)_n}\!-\!\frac{c\left(-aq,-bq\right)_{\infty}}{ab\left(-cq\right)_{\infty}}\!\sum_{n\ge 1}\!\frac{\pa{ab}{c^2}^nq^{n^2}}{\left(\frac{aq}{c},\frac{bq}{c}\right)_n}.
\end{align*}
\end{lemma}
\noindent The second identity of Ramanujan was given in  \cite[equation (3.18)$_\text{R}$]{A0}.

\begin{lemma}\label{lemident2}
We have
\begin{align*}
\sum_{n\ge0}\frac{(cq)_nq^n}{(-aq,-bq)_n}\!=\!\left(1+a^{-1}\right)\!\sum_{n\ge0}\!\frac{(cq)_n\left(-\frac{b}{a}\right)^n\!q^{\frac{n(n+1)}{2}}}{\left(-\frac{cq}{a}\right)_{n+1}(-bq)_n}\!-\!\frac{a^{-1}(cq)_{\infty}}{(-aq,-bq)_{\infty}}\!\sum_{n\ge0}\!\frac{\left(-\frac{b}{a}\right)^n\!q^{\frac{n(n+1)}{2}}}{\left(-\frac{cq}{a}\right)_{n+1}}.
\end{align*}
\end{lemma}


Recall the following two identities from \cite[equation (2.2.5) and equation (2.2.6)]{Andbook}. 
\begin{lemma}\label{Euler}
For $|t|, |q|<1$, we have 
\[
\sum_{n\ge 0}\frac{t^n}{(q)_n}=\frac{1}{(t)_{\infty}}\ \ \text{and}\ \ \sum_{n\ge 0}\frac{t^n q^{\frac{n(n-1)}{2}}}{\left(q\right)_n}=\left(-t\right)_{\infty}.
\]	
\end{lemma} 



Next, we require an identity due to Garvan \cite[Theorem 1.2]{G}.
\begin{lemma}\label{Garvan}
	We have
	\begin{equation*}
	\sum_{n\ge 1}\frac{(-1)^{n+1}\zeta^n q^{n^2}}{\left(1-\zeta q^{2n}\right)\left(\zeta q;q^2\right)_n }=\sum_{n\ge 1}\frac{\left(q\right)_{n-1}\zeta^n q^{\frac{n(n+1)}{2}}}{\left(\zeta q\right)_n}.
	\end{equation*}
\end{lemma}
 Due to Andrews \cite[equation (4.4)]{Kang}, we have the following:
\begin{lemma}\label{kang}
We have
\begin{equation*}
\sum_{n\ge 0}\frac{q^{\frac{n(n+1)}{2}}}{\left(-q\right)_{n+1}}=1.
\end{equation*}	
\end{lemma}
In \cite[Theorem 1.1]{donato}, we find the following identity.
\begin{lemma}\label{Donato}
We have
\begin{equation*}
 \sum_{n\ge 0} \frac{nq^{\frac{n(n-1)}{2}}}{(-q)_n} =\sigma(q). 
\end{equation*}	
\end{lemma}
From \cite[equation (2.3.1)]{AB}, we have the following lemma.
\begin{lemma}\label{Lostntbk}
We have, for $a\in \mathbb{C}$ and $|bq|<1$,
\[
\sum_{n\ge 0}\frac{(ab)^n q^{n^2}}{\left(aq,bq\right)_n}=1+a\sum_{n\ge 1}\frac{(bq)^n}{\left(aq\right)_n}.
\]	
\end{lemma}
From \cite[Corollary 2.4]{Andbook}, we record the following identity.
\begin{lemma}\label{Heine}
For $|c|<|ab|, |q|<1$, we have 
\[
\sum_{n\ge 0}\frac{(a,b)_n\left(\frac{c}{ab}\right)^n}{(q,c)_n}=\frac{\left(\frac{c}{a}, \frac{c}{b}\right)_{\infty}}{\left(c, \frac{c}{ab}\right)_{\infty}}.
\]
\end{lemma}
From \cite[Example 10, p. 29]{Andbook}, we record the following identity.
\begin{lemma}\label{Andbook}
We have
\begin{equation*}
1+\sum_{n\ge 1}(-1)^n\left(\zeta^{3n-1}q^{\frac{n(3n-1)}{2}}+\zeta^{3n}q^{\frac{n(3n+1)}{2}}\right)=\sum_{n\ge 0}\frac{(-1)^n\zeta^{2n}q^{\frac{n(n+1)}{2}}}{\left(\zeta q\right)_n}.
\end{equation*}	
\end{lemma}
We also require identities by Bhoria, Eyyunni, and Maji. The first is (see \cite[Theorem 2.1]{BEM}).
\begin{lemma}\label{BEMeqn}
We have, for $|ad|$, $|cq|<1$, 
\begin{align*}
\sum_{n\ge1}\frac{\left(\frac{b}{a},\frac{c}{d}\right)_n (ad)^n}{(b,cq)_n} =\frac{(a-b)(d-c)}{ad-b}\sum_{n\ge0}\frac{\left(a,\frac{bd}{c}\right)_nc^n}{(b,ad)_n}\left(\frac{adq^n}{1-adq^n}-\frac{bq^n}{1-bq^n}\right).
\end{align*}
\end{lemma}
\noindent The second identity (see \cite[Theorem 2.2]{BEM}) is the following:
\begin{lemma}\label{BEMeqn1}
For $|cq|<1$,
\[
\sum_{n\ge 1}\frac{\left(\frac{c}{d}\right)_n\left(-\z d\right)^nq^{\frac{n(n+1)}{2}}}{\left(\z q,cq\right)_n}=\frac{\z \left(c-d\right)}{c}\sum_{n \ge 1}\frac{\left(\frac{\z dq}{c}\right)_{n-1}\left(cq\right)^n}{\left(\z q\right)_n}.
\]	
\end{lemma}
Next, we list the following identity from  \cite[equation (3.2.6)]{Berndt}.
\begin{lemma}\label{AGLeqn}
We have 
\begin{align*}
1+4\sum_{n\ge0}\frac{(-1)^{n+1}q^{2n+1}}{1+q^{2n+1}}=\frac{(q)^2_{\infty}}{(-q)^2_{\infty}}.
\end{align*}	
\end{lemma}
Using \cite[equation (4.5)]{Agarwal} with $q\mapsto q^2, \a= -q, \b= -\z q^3$, and $t= -\z q^2$, we get the following. 
 \begin{lemma}\label{Agarwal}
 We have 
 \begin{align*}
 &\sum_{n\ge 0}\frac{\left(-q;q^2\right)_n(-1)^n\z^nq^{2n}}{\left(-\z q^3;q^2\right)_n}
 =\frac{\left(q,\z q^2\right)_{\infty}\left(\frac{1}{\z q};q^2\right)_{\infty}}{2\left(-\z q^2\right)_{\infty}}
 +\!\frac{1}{2}\left(1+\frac{1}{\z q}\right)\sum_{n\ge 0}\frac{\left(-\frac{1}{\z};q^2\right)_n(-1)^n q^{n}}{\left(-q^2;q^2\right)_n}.
 \end{align*}	
 \end{lemma}

Finally, we require the Rogers--Fine identity \cite[equation (14.1)]{Fine}.
\begin{lemma}\label{RF}
 We have 
\[
\sum_{n\ge 0}\frac{\left(\a\right)_nt^n}{\left(\b\right)_n}=\sum_{n\ge 0}\frac{\left(\a, \frac{\a tq}{\b}\right)_n\left(1-\a tq^{2n}\right)\left(\frac{\b t}{q}\right)^n q^{n^2}}{\left(\b\right)_n\left(t\right)_{n+1}}.
\]	
\end{lemma}

\subsection{A Tauberian theorem} We also require a modified version of Ingham's Tauberian Theorem \cite{I}, stated in \cite[Theorem 1.1]{BJM}.  
For $\Delta\geq0$ define
\begin{equation*}
R_\Delta := \{x+iy:\, x,y\in\R,\, x>0,\, |y|\le\Delta x\}.
\end{equation*}	
\begin{proposition}\label{P:CorToIngham}
	Let $B(q)=\sum_{n\ge0}b(n)q^n$ be a power series whose radius of convergence is at least one and assume that the $b(n)$ are non-negative and weakly increasing.
	Also suppose that $\l$, $\b$, $\g\in\R$ with $\g>0$ exist such that
	\begin{equation}\label{E:as1}
	B\left(e^{-t}\right) \sim \l t^\b e^\frac\g t \quad\text{as } t \to 0^+,\qquad B\left(e^{-z}\right) \ll |z|^\b e^\frac{\g}{|z|} \quad\text{as } z \to 0,
	\end{equation}
	with the latter condition holding in each region $R_\Delta$
	for $\Delta \geq 0$. Then we have
	\[
	b(n) \sim \frac{\l\g^{\frac\b2+\frac14}}{2\sqrt\pi n^{\frac\b2+\frac34}}e^{2\sqrt{\g n}} \qquad\text{as } n \to \infty.
	\]
\end{proposition}
Moreover, we require the asymptotic behavior of $(q)_\infty$ as $q\to 1$.
\begin{lemma}\label{L:q-Pochhammer asymptotics}
	Let $q=e^{-z}$ and $\Delta \geq 0$. Then, as $z\to0$ in $R_\Delta$, we have
	\begin{align*}
	\left( q \right)_\infty &\sim \sqrt{\dfrac{2\pi}{z}} e^{- \frac{\pi^2}{6z}}.
	\end{align*}
\end{lemma}


\subsection{The Euler--Maclaurin summation formula} We use a variation of the Euler--Maclaurin summation formula, an important tool for computing asymptotic expansions of infinite sums, studied by Zagier \cite[Proposition 3]{Z} for one-dimensional smooth functions and extended by Jennings-Shaffer, Mahlburg, and the second author \cite[Theorem 1.2]{BJM} for multi-dimensional complex-valued functions. A function $g$, defined over an unbounded domain $D\subset \mathbb{C}$, is of {\it sufficient decay} if there exists $\varepsilon>0$ such that $g(w)\ll |w|^{-1-\varepsilon}$ uniformly as $|w|\to \infty$ in $D$.  For $0\le \theta<\frac{\pi}{2}$, we define
$D_{\theta}:=\{re^{i\a}: r\ge 0\ \text{and}\ |\a|\le \theta\}$. 
\begin{proposition}\label{EM1}
Let $0\le \theta<\frac{\pi}{2}$ and $g$ be a holomorphic function in a domain containing $D_{\theta}$. Also assume that $g$ and all of its derivatives are of sufficient decay. Then, for $a\in \mathbb{R}_{\ge 0}$, $N\in \mathbb{N}_0$, and as $z\to 0$ uniformly in $D_{\theta}$, we have
\[
\sum_{m\ge 0}g\left((m+a)z\right)=\frac 1z\int_{0}^{\infty} g(w)dw-\sum_{n=0}^{N-1}\frac{B_{n+1}(a)g^{(n)}(0)}{(n+1)!}z^n+O\left(z^N\right),
\]	
where $B_n(x)$ denotes the $n$-th Bernoulli polynomial.
\end{proposition}
Next, we consider two-dimensional sums. Throughout we write $\bm{x}=\left(x_1,x_2\right)\in \mathbb{R}^2$. A function $f$ defined on an unbounded domain $D\subset\mathbb{C}^2$ is of {\it sufficient decay} if there exist $\varepsilon_1,\varepsilon_2>0$ such that $f(\bm{x})\ll |w_1|^{-1-\varepsilon_1}|w_2|^{-1-\varepsilon_2}$ uniformly as $|w_1|+|w_2|\to \infty$ in $D$. From \cite[Proposition 5.2]{BJM} (with $r=2$), we have the following. 
\begin{proposition}\label{EM2}
Let $0\le \theta<\frac{\pi}{2}$ and $f$ be a holomorphic function in a domain containing $D_{\theta}\times D_{\theta}$. Also assume that $f$ and all of its derivatives are of sufficient decay. Then for $\bm{a}\in \mathbb{R}^2_{\ge 0}$, $N\in \mathbb{N}_0$, and as $z\to 0$ uniformly in $D_{\theta}$, we have
\begin{align*}
\sum_{m\ge 0}f\left(\left(\bm{m}+\bm{a}\right)z\right)&=\frac{1}{z^2}\int_{0}^{\infty}\int_{0}^{\infty} f(\bm{x})dx_1dx_2-\frac 1z\sum_{n_1=0}^{N}\frac{B_{n_1+1}\left(a_1\right)z^{n_1}}{\left(n_1+1\right)!}\int_{0}^{\infty}f^{\left(n_1,0\right)}\left(0,x_2\right)dx_2\\
&\hspace{1 cm}-\frac 1z\sum_{n_2=0}^{N}\frac{B_{n_2+1}\left(a_2\right)z^{n_2}}{\left(n_2+1\right)!}\int_{0}^{\infty}f^{\left(0,n_2\right)}\left(x_1,0\right)dx_1\\
&\hspace{1 cm}+\sum_{n_1+n_2<N}\frac{B_{n_1+1}\left(a_1\right)B_{n_2+1}\left(a_2\right)f^{\left(n_1,n_2\right)}\left(\bm{0}\right)}{\left(n_1+1\right)!\left(n_2+1\right)!}+O\left(z^N\right).
\end{align*}
\end{proposition}


\section{Proof of Theorem \ref{thm1} and Corollary \ref{thm1a}}\label{genfn}


In this section, we prove \Cref{thm1} and \Cref{thm1a}. We start with \Cref{thm1}.

\noindent\emph{Proof of Theorem \ref{thm1}}. By \eqref{pod}, we have 
\begin{align}\label{eqn2}
\overline{P}_{\rm od}(q)
=\!\frac{q\left(-q;q^2\right)_{\infty}}{\left(q^2,q^3;q^2\right)_{\infty}}\sum_{n\ge 0}\!\frac{\left(q^2,q^3;q^2\right)_{n}q^{2n}}{\left(-q^3;q^2\right)_{n}}.
\end{align}
Changing $q\mapsto q^2$ in \Cref{lemident1} and then dividing both sides by $q^2$, we get
\begin{align*}
&\sum_{n\ge 0}\frac{\left(-aq^2,-bq^2;q^2\right)_nq^{2n}}{\left(-cq^2;q^2\right)_n}\\[-7pt]
&\hspace{3 cm}=\frac{c}{abq^2}\sum_{n\ge 1}\frac{\left(-c^{-1};q^2\right)_{\!n} \pa{ab}{c}^{\!n} q^{n(n+1)}}{\left(\frac{aq^2}{c},\frac{bq^2}{c};q^2\right)_{\!n}} \!-\!\frac{c\left(-aq^2,-bq^2;q^2\right)_{\!\infty}}{abq^2\left(-cq^2;q^2\right)_{\!\infty}}\sum_{n\ge 1}\frac{\pa{ab}{c^2}^nq^{2n^2}}{\left(\frac{aq^2}{c},\frac{bq^2}{c};q^2\right)_{\!n}}.
\end{align*}
Setting $a=-1,\ b=-q$, and $c=q$, yields
\begin{align*}
\sum_{n\ge 0}\frac{\left(q^2,q^3;q^2\right)_nq^{2n}}{\left(-q^3;q^2\right)_n}&=
\frac{1+q}{q^3}\sum_{n\ge 1}\frac{q^{n(n+1)}}{\left(1+q^{2n-1}\right) \left(-q^2;q^2\right)_n}-\frac{\left(q^2,q^3;q^2\right)_{\infty}}{q^2\left(-q^3;q^2\right)_{\infty}}\sum_{n\ge 1}\frac{q^{2n^2-n}}{\left(-q\right)_{2n}}.
\end{align*}
Plugging this into \eqref{eqn2}, we obtain
\begin{align}\label{eqn5}
\overline{P}_{\rm od}(q)
&=\frac{(1+q)\left(-q;q^2\right)_{\infty}}{q^2\left(q^2,q^3;q^2\right)_{\infty}}\sum_{n\ge 1}\frac{q^{n(n+1)}}{\left(1+q^{2n-1}\right) \left(-q^2;q^2\right)_n}-\frac{1+q}{q}\sum_{n\ge 1}\frac{q^{2n^2-n}}{\left(-q\right)_{2n}}.
\end{align}
Letting $\zeta=-q^{-1}$ in \Cref{Garvan}, we get 
\begin{align}\label{eqn6}
\sum_{n\ge 1}\frac{q^{n(n-1)}}{\left(1+ q^{2n-1}\right)\left(-q^2;q^2\right)_{n-1} }
&=1+\sum_{n\ge 1}\frac{(-1)^n \left(q\right)_{n}q^{\frac{n(n+1)}{2}}}{\left(-q\right)_{n}}.
\end{align}
 Continuing with the left-hand side of \eqref{eqn6} and writing $1+q^{2n}=1+q^{2n-1}-(1-q)q^{2n-1}$, we have
\begin{align}\label{eqn7}
\sum_{n\ge 1}\frac{q^{n(n-1)}}{\left(1+ q^{2n-1}\right)\left(-q^2;q^2\right)_{n-1} }
=\sum_{n\ge 1}\frac{q^{n(n-1)}}{\left(-q^2;q^2\right)_{n}}-\frac{1-q}{q}\sum_{n\ge 1}\frac{q^{n(n+1)}}{ \left(1+ q^{2n-1}\right)\left(-q^2;q^2\right)_{n}}.
\end{align}
Plugging \eqref{eqn6} into \eqref{eqn7} and then applying \Cref{kang} with $q\mapsto q^2$, we obtain 
\begin{equation}\label{above3.6}
\sum_{n\ge 1}\frac{q^{n(n+1)}}{\left(1+ q^{2n-1}\right)\left(-q^2;q^2\right)_{n} }=-\frac{q}{1-q}\sum_{n\ge 1}\frac{(-1)^n \left(q\right)_{n}q^{\frac{n(n+1)}{2}}}{\left(-q\right)_{n}}.
\end{equation}
Combining this with \eqref{eqn5} and \eqref{corson}, we have
\begin{align*}
\overline{P}_{\rm od}(q)\!&=\!-\frac{(1+q)\left(-q;q^2\right)_{\infty}}{q\left(q\right)_{\infty}}W_1(q)\!+\!\frac{(1+q)\left(-q;q^2\right)_{\infty}}{q\left(q\right)_{\infty}}-\!\frac{1+q}{q}\sum_{n\ge 1}\frac{q^{2n^2-n}}{(-q)_{2n}}.
\end{align*}

Thus, to conclude the proof, we need to show that
\begin{equation}\label{Neweqn2paper}
\sum_{n\ge 1}\frac{q^{2n^2-n}}{\left(-q\right)_{2n}}=\sum_{n\in \mathbb{Z}}\textnormal{sgn}(n)q^{\frac{n(3n-1)}{2}}.
\end{equation}
First note that
\[
\sum_{n\in \mathbb{Z}}\textnormal{sgn}(n)q^{\frac{n(3n-1)}{2}}=\sum_{n\ge 1}\left(1-q^n\right)q^{\frac{n(3n-1)}{2}}=-\sum_{n\ge 1}\frac{(-1)^nq^{\frac{n(n+1)}{2}}}{\left(-q\right)_n},
\]
by setting $\zeta=-1$ in \Cref{Andbook}. 
To prove \eqref{Neweqn2paper}, it thus suffices to show that
\begin{equation*}
\sum_{n\ge 1}\frac{q^{2n^2-n}}{\left(-q\right)_{2n}}=-\sum_{n\ge 1}\frac{(-1)^nq^{\frac{n(n+1)}{2}}}{\left(-q\right)_n}.
\end{equation*}
To show this, we split the left-hand side depending on the parity of $n$ and then use \Cref{kang}.\qed 

We now determine the asymptotics of $\overline{p}_{\od}(n)$ as $n\to\infty$. For this, we prove the monotonicity of $\overline{p}_{\od}(n)$. For $A(q)=\sum_{n\ge 0}a(n)q^n$, by $A(q)\succeq 0$, we mean that $a(n)\ge 0$ for $n\in \mathbb{N}_0$. 
\begin{lemma}\label{monotone1}
	For $n\in \mathbb{N}_0$, we have 
	 $\overline{p}_{\od}(n+1)\ge \overline{p}_{\od}(n)$.
\end{lemma}

\begin{proof}
Proving $\overline{p}_{\od}(n+1)\ge \overline{p}_{\od}(n)$ for $n\in \mathbb{N}_0$ is equivalent to showing that $(1-q)\overline{P}_{\textup{od}}(q)\succeq 0$. By \eqref{eqn2}, we have 
\begin{align*}
\overline{P}_{\textup{od}}(q)
&=q\left(1+q\right)\sum_{n\ge 0}\frac{\left(-q^{2n+3};q^2\right)_{\infty} q^{2n}}{\left(q^{2n+2}\right)_{\infty}}.
\end{align*}
Thus, to prove $(1-q)\overline{P}_{\textup{od}}(q)\succeq 0$, it suffices to show that
\begin{equation}\label{monotone1eqn1}
\left(1-q^2\right)\sum_{n\ge 0}\frac{\left(-q^{2n+3};q^2\right)_{\infty} q^{2n}}{\left(q^{2n+2}\right)_{\infty}}=\frac{\left(-q^3;q^2\right)_{\infty}}{\left(q^3\right)_{\infty}}+\left(1-q^2\right)\sum_{n\ge 1}\frac{\left(-q^{2n+3};q^2\right)_{\infty} q^{2n}}{\left(q^{2n+2}\right)_{\infty}}\succeq 0.
\end{equation}
Note that $\frac{(-q^3;q^2)_{\infty}}{(q)^3_{\infty}}\succeq 0$. Note that \eqref{monotone1eqn1} follows if we show that, for $n\in \mathbb{N}$, 
\begin{equation}\label{monotone1eqn2}
\left(1-q^2\right)\frac{\left(-q^{2n+3};q^2\right)_{\infty} }{\left(q^{2n+2}\right)_{\infty}}-1+q^2\succeq 0.
\end{equation}
Indeed using this, we see that 
\[
\left(1-q^2\right)\sum_{n\ge 1}\frac{\left(-q^{2n+3};q^2\right)_{\infty} q^{2n}}{\left(q^{2n+2}\right)_{\infty}}\succeq \left(1-q^2\right)\sum_{n\ge 1}q^{2n}
=q^2\succeq 0.
\]

To prove \eqref{monotone1eqn2}, we use \Cref{Euler}, to write
\begin{align*}
\left(1-q^2\right)\frac{\left(-q^{2n+3};q^2\right)_{\infty} }{\left(q^{2n+2}\right)_{\infty}}-1+q^2
&=\left(1-q^2\right)\left(\sum_{r,s\ge 0}\frac{q^{r^2+\left(2n+2\right)(r+s)}}{\left(q^2;q^2\right)_r\left(q\right)_s}-1\right)\\
&=\sum_{r\ge 1}\frac{q^{r^2+\left(2n+2\right)r}}{\left(q^4;q^2\right)_{r-1}}+\left(1+q\right)\sum_{\substack{r\ge 0\\s\ge 1}}\frac{q^{r^2+\left(2n+2\right)(r+s)}}{\left(q^2;q^2\right)_r\left(q^2\right)_{s-1}}\succeq 0.
\end{align*}
This finishes the proof.
\end{proof}



\begin{proof}[Proof of \Cref{thm1a}] 

We first note that $\overline{p}_{\od}(n)$ is non-negative because it counts restricted overpartitions and by \Cref{monotone1}, it is weakly increasing. Thus it remains to determine the asymptotics of $\overline{P}_{\od}(q)$. With $q=e^{-z}$ and as $z\to 0$ in $R_{\Delta}$, we show that
\begin{equation}\label{boundW}
\overline{P}_{\textup{od}}(q)\sim \frac{z^{\frac 32}e^{\frac{5\pi^2}{24z}}}{\sqrt{2\pi}}.
\end{equation}
We split $\overline{P}_{\od}(q)$ as in \Cref{thm1}. Using \Cref{EM1}, the final term can be bounded as  
\[
-\!\frac{1+q}{q}\sum_{n\in \mathbb{Z}}\textnormal{sgn}(n)q^{\frac{n(3n-1)}{2}}\ll \frac{1}{\sqrt{z}}.
\]
The rest can be written as 
\[
\frac{(1+q)\left(-q;q^2\right)_{\infty}}{q(q)_{\infty}}\left(1-W_1(q)\right).
\]
Using \Cref{L:q-Pochhammer asymptotics}, we have, as $z\to 0$ in $R_{\Delta}$, 
\begin{equation*}
\frac{(1+q)\left(-q;q^2\right)_{\infty}}{q(q)_{\infty}}\sim 
\sqrt{\frac{2z}{\pi}}e^{\frac{5\pi^2}{24z}}.
\end{equation*}

So to prove \eqref{boundW}, we need to show that 
\begin{equation}\label{boundW1}
1-W_1(q)\sim \frac z2.
\end{equation}
From \cite[Theorem 2.3]{CFLZ}, we have
\begin{align}\label{boundW2}
W_1(q)
&=\sum_{\substack{n\ge 0}}(-1)^{n}q^{2n^2+n}\left(1-q^{2n+1}\right)+2\sum_{\substack{n\ge 0\\1\le j\le n}}(-1)^{n+j}q^{2n^2+n-j^2}\left(1-q^{2n+1}\right).
\end{align}
The first summand on the right-hand side of \eqref{boundW2} may be written as 
\begin{align}\label{boundW3}
&q^{-\frac 18}\sum_{n\ge 0}\left(q^{2\left(2n+\frac 14\right)^2}-q^{2\left(2n+\frac 34\right)^2}-q^{2\left(2n+\frac 54\right)^2}+q^{2\left(2n+\frac 74\right)^2}\right)\nonumber\\
&\hspace{1 cm}=q^{-\frac 18}\sum_{j\in\{1,3,5,7\}}\leg{2}{j}\sum_{n\ge 0}q^{8\left(n+\frac j8\right)^2}
=e^{\frac z8}\sum_{j\in\{1,3,5,7\}}\leg{2}{j}\sum_{n\ge 0}g\left(\left(n+\frac j8\right)\sqrt{z}\right),
\end{align}
where $g(x):=e^{-8x^2}$. Applying \Cref{EM1} with $N=3$ and $z\mapsto \sqrt{z}$, we have 
\begin{align}\label{boundW4}
&\sum_{j\in\{1,3,5,7\}}\leg{2}{j}\sum_{n\ge 0}g\left(\left(n+\frac j8\right)\sqrt{z}\right)\nonumber\\[-7 pt]
&\hspace{3 cm}=\sum_{j\in\{1,3,5,7\}}\!\!\leg{2}{j}\Biggl(\frac{1}{\sqrt{z}}\int_{0}^{\infty}\!\!\!g(w)dw-\!\!\sum_{n=0}^{3}\frac{B_{n+1}\left(\frac j8\right)g^{(n)}(0) z^{\frac n2}}{(n+1)!}+O\left(z^2\right)\Biggr)\nonumber\\
&\hspace{3 cm}=-\!\!\!\!\!\!\sum_{j\in\{1,3,5,7\}}\!\leg{2}{j}\sum_{n=0}^{1}\frac{B_{2n+1}\left(\frac j8\right)g^{\left(2n\right)}(0) z^{n}}{(2n+1)!}+O\left(z^2\right)
=O\left(z^{2}\right).
\end{align}
Here we use that $\sum_{j\in \{1,3,5,7\}}\leg{2}{j}=0$, note that $g$ is an even function, and 
compute 
\[
\sum_{j\in \{1,3,5,7\}}\leg{2}{j}B_{2n+1}\pa{j}{8}
=
0\ \ \text{for}\ n\in\{0,1\}.
\]

Next, shifting $n\mapsto n+j$ and then $j\mapsto j+1$, the second sum in \eqref{boundW2} becomes 
\begin{align*}
&2q^{-\frac 18}\sum_{n,j\ge 0}(-1)^n\left(q^{2\left(n+\frac 14\right)^2+4\left(n+\frac 14\right)\left(j+1\right)+\left(j+1\right)^2}-q^{2\left(n+\frac 34\right)^2+4\left(n+\frac 34\right)\left(j+1\right)+\left(j+1\right)^2}\right)\nonumber\\
&\hspace{2 cm}=2q^{-\frac 18}\sum_{\d,r\in\{0,1\}}(-1)^{\d+r}\sum_{n,j\ge 0}q^{2\left(2n+\d+\frac{1+2r}{4}\right)^2+4\left(2n+\d+\frac{1+2r}{4}\right)(j+1)+(j+1)^2}\nonumber\\
&\hspace{2 cm}
=2e^{\frac z8}\sum_{\d,r\in\{0,1\}}(-1)^{\d+r}\sum_{n,j\ge 0}f\left(\left(n+\frac{\d}2+\frac{1+2r}8,j+1\right)\sqrt{z}\right),
\end{align*}
where $f(x,y):=e^{-8x^2-8xy-y^2}$. Applying \Cref{EM2} yields
\begin{align*}
&2e^{\frac z8}\sum_{\d,r\in\{0,1\}}(-1)^{\d+r}\Biggl(\frac 1z\int_{0}^{\infty}\int_{0}^{\infty}f(x,y) dx dy-\frac{1}{\sqrt{z}}\Biggl(\sum_{n=0}^{N}\frac{B_{n+1}\left(\frac{\d}{2}+\frac{1+2r}{8}\right)z^{\frac n2}}{(n+1)!}\int_{0}^{\infty}f^{(n,0)}(0,y) dy\\
&+\sum_{m=0}^{N}\frac{B_{m+1}\left(1\right)z^{\frac m2}}{(m+1)!}\int_{0}^{\infty}f^{(0,m)}(x,0) dx\Biggr)+\sum_{n+m<N}\frac{B_{n+1}\left(\frac{\d}{2}+\frac{1+2r}{8}\right)B_{m+1}(1)}{(n+1)!(m+1)!}f^{(n,m)}(0,0) z^{\frac{n+m}{2}}\\[-10 pt]
&\hspace{13.5 cm}+O\left(z^{\frac N2}\right)\Biggr).
\end{align*}
Now those parts cancel that are independent of $\d,r$ because of the $(-1)^{\d+r}$. Moreover $f^{(n,m)}(0,0)=0$ unless $n+m$ is even. Thus with $N=4$, the above becomes  
\begin{align}\label{boundW8}
&2e^{\frac z8}\sum_{\d,r\in\{0,1\}}(-1)^{\d+r}\vast( -\sum_{n=0}^{4}\frac{B_{n+1}\left(\frac{\d}{2}+\frac{1+2r}{8}\right)\!z^{\frac{n-1}2}}{(n+1)!}\int_{0}^{\infty}f^{(n,0)}(0,y) dy\nonumber\\[-13 pt]
&\hspace{4 cm}\left.+\sum_{\substack{n+m\le 3\\ 2\mid (n+m)}}\frac{B_{n+1}\left(\frac{\d}{2}+\frac{1+2r}{8}\right)B_{m+1}(1)}{(n+1)!(m+1)!}f^{(n,m)}(0,0) z^{\frac{n+m}{2}}\right)+O\left(z^{2}\right).
\end{align}
We compute 
\begin{equation}\label{boundW9}
\sum_{\d,r\in\{0,1\}}(-1)^{\d+r}B_{n+1}\left(\frac{\d}{2}+\frac{1+2r}{8}\right)=\begin{cases}
0 &\ \text{if}\ n \ \text{is even},\\
\frac 14 &\ \text{if}\ n=1,\\
-\frac{11}{128} &\ \text{if}\ n=3.
\end{cases}
\end{equation}
Note that $n$ odd, $0\le n+m\le 3$, and $2\mid\!(n+m)$ implies that $n=m=1$. Thus \eqref{boundW8} 
 becomes 
\begin{align}\label{boundW9a}
&e^{\frac z8}\sum_{\d,r\in\{0,1\}}(-1)^{\d+r}\Biggl(-B_{2}\left(\frac{\d}{2}+\frac{1+2r}{8}\right)\int_{0}^{\infty}f^{\left(1,0\right)}(0,y) dy-\frac{B_{4}\left(\frac{\d}{2}+\frac{1+2r}{8}\right)\!z}{12}\int_{0}^{\infty}f^{\left(3,0\right)}(0,y) dy\nonumber\\
&\hspace{7 cm}+\frac{B_{2}\left(\frac{\d}{2}+\frac{1+2r}{8}\right)}{12}f^{(1,1)}(0,0) z\Biggr)+O\left(z^{2}\right).
\end{align}
Computing
\[
\int_{0}^{\infty}f^{\left(1,0\right)}(0,y) dy=-4,\ \int_{0}^{\infty}f^{\left(3,0\right)}(0,y) dy=-64,\ \text{and}\ \ f^{(1,1)}(0,0)=-8,
\]
and using \eqref{boundW9}, \eqref{boundW9a} becomes $1-\frac z2+O(z^2)$.
This implies that
\[
2\sum_{\substack{n\ge 0\\1\le j\le n}}(-1)^{n+j}q^{2n^2+n-j^2}\left(1-q^{2n+1}\right)=1-\frac z2+O\left(z^2\right).
\]
Combining this with \eqref{boundW4}, \eqref{boundW3}, 
 \eqref{boundW2}, and using \Cref{P:CorToIngham}, we obtain the claim.\qedhere  
\end{proof}


\section{Proof of Theorem \ref{thm2} and Corollary \ref{thm6.2}}\label{sec5}

First we prove Theorem \ref{thm2}. For this, we need the following lemma, which is obtained by using \Cref{BEMeqn} and \Cref{AGLeqn}.

\begin{lemma}\label{theta squared}
We have
\begin{align*}
\sum_{n\ge 1}\frac{(-1)^n \left(q^2;q^2\right)_{n-1} q^{n^2}}{\left(-q\right)_{2n}}=\frac{1}{4}\left(\frac{(q)_{\infty}^2}{(-q)_{\infty}^{2}}-1\right).
\end{align*}
\end{lemma}

\begin{proof}[Proof of \Cref{thm2}] 
Using \eqref{thm2eqn0}, we have
\begin{align}\label{thm2eqn1}
\overline{P}_{\text{ev}}(q)
&=\frac{q}{(q)_{\infty}}\sum_{n\ge 0}\frac{\left(q^2;q^2\right)_{n}q^{2n}}{\left(-q^2,-q^3;q^2\right)_{n}}.
\end{align}
Changing $q\mapsto q^2$ and then setting $a=c=1$ and $b=q$ in \Cref{lemident2}, we get 
\begin{align*}
\sum_{n\ge 0}\frac{\left(q^2;q^2\right)_{n}q^{2n}}{\left(-q^2,-q^3;q^2\right)_{n}}=2\sum_{n\ge 0}\frac{(-1)^n \left(q^2;q^2\right)_n q^{n^2+2n}}{\left(-q^2;q^2\right)_{n+1}\left(-q^3;q^2\right)_n}-(1+q)(q)_\infty\sum_{n\ge 0}\frac{(-1)^n q^{n^2+2n}}{\left(-q^2;q^2\right)_{n+1}}.
\end{align*}
Applying this to \eqref{thm2eqn1}, using \eqref{Ramanujanphi} and then invoking Lemma \ref{theta squared}, we conclude the proof.\qedhere
\end{proof}





Next, we determine the asymptotics of $\overline{p}_{\textup{ev}}(n)$. For this, we prove its  monotonicity. 
\begin{lemma}\label{monotone2}
 We have, 
  for $n\in \mathbb{N}_0$, $\overline{p}_{\textnormal{ev}}(n)\le \overline{p}_{\textnormal{ev}}(n+1)$.
\end{lemma}
\begin{proof}
By \eqref{thm2eqn1}, we have  
\begin{align}\label{genfunc2}
\left(1-q\right)\overline{P}_{\textup{ev}}(q)
=\frac{q}{\left(q^2\right)_{\infty}}+\left(1-q^2\right)\sum_{n\ge 0}\frac{\left(-q^{2n+4}\right)_{\infty}q^{2n+3}}{\left(q^{2n+4};q^2\right)_{\infty}}.
\end{align}
To prove the monotonicity of $\overline{p}_{\textnormal{ev}}(n)$, we need to show that $(1-q)\overline{P}_{\textup{ev}}(q)\succeq0$. 
 Since $\frac{q}{(q^2)_{\infty}}\succeq 0$, by \eqref{genfunc2}, this follows if we show that for $n\in \mathbb{N}_0$, 
\begin{equation*}
\left(1-q^2\right)\frac{\left(-q^{2n+4}\right)_{\infty}}{\left(q^{2n+4};q^2\right)_{\infty}}-1+q^2\succeq 0.
\end{equation*}
To prove this, using \Cref{Euler}, we have 
\begin{align*}
\left(1-q^2\right)\frac{\left(-q^{2n+4}\right)_{\infty} }{\left(q^{2n+4};q^2\right)_{\infty}}-1+q^2
&=\left(1-q^2\right)\left(\sum_{r,s\ge 0}\frac{q^{\frac{r(r-1)}{2}+\left(2n+4\right)(r+s)}}{\left(q\right)_r\left(q^2;q^2\right)_s}-1\right)\\
&=\left(1+q\right)\sum_{r\ge 1}\frac{q^{\frac{r(r-1)}{2}+\left(2n+4\right)r}}{\left(q^2\right)_{r-1}}+\sum_{\substack{r\ge 0\\s\ge 1}}\frac{q^{\frac{r(r-1)}{2}+\left(2n+4\right)(r+s)}}{\left(q\right)_r\left(q^4;q^2\right)_{s-1}}\succeq 0.
\end{align*}
This proves the claim.
\end{proof}

\begin{proof}[Proof of \Cref{thm6.2}] 
By \Cref{monotone2}, $\overline{p}_{\textnormal{ev}}(n)$ is non-negative and  weakly increasing. Therefore, to apply \Cref{P:CorToIngham}, it remains to verify that $\overline{P}_{\textnormal{ev}}(q)$ satisfies \eqref{E:as1}. We show that
\begin{equation}\label{phi}
\overline{P}_{\textup{ev}}(q)\sim  \sqrt{\frac{z}{2\pi}} e^{\frac{\pi^2}{6z}}\ \ \text{with}\ q=e^{-z},\ \text{as}\ z\to 0\ \text{in}\ R_{\Delta}.
\end{equation}  
Using \Cref{L:q-Pochhammer asymptotics}, we have, as $z\to 0$ in $R_{\Delta}$, 
\[
\frac{1+q}{2(q)_{\infty}}\left(1-\frac{(q)^2_{\infty}}{(-q)^2_{\infty}}\right)\sim \sqrt{\frac{z}{2\pi}}e^{\frac{\pi^2}{6z}}.
\]	
Thus, by \Cref{thm2}, \eqref{phi} follows if we show that
\[
\left(1+q\right)\left(\phi(-q)-1\right)
\ll e^{\frac{\pi^2}{24z}}.
\]
This indeed holds by \cite{BCN} and \Cref{L:q-Pochhammer asymptotics}. 
Applying Proposition \ref{P:CorToIngham}, we conclude the proof. 
\end{proof}

\section{Proof of Theorems \ref{Thm1}, \ref{newthm1}, \ref{thm4}, and \ref{thm3}}\label{z:Gen}


\begin{proof}[Proof of \Cref{Thm1}] Using the generating function \eqref{zgen2}, and then applying \Cref{lemident1} with $q\mapsto q^2$, $a=-\z$, $b=-q$, and $c=q$, we obtain
\begin{multline}\label{Thm1eqn1prf}
P_{\textup{od}}(\z;q)\!=\!\frac{\left(1+q\right)\left(-q^3;q^2\right)_{\infty}}{\zeta q\left(\z q^2,q^3;q^2\right)_{\infty}}\\
\times\left(\sum_{n\ge 1}\frac{\left(-\frac 1q;q^2\right)_n\z^nq^{n(n+1)}}{\left(-\z q,-q^2;q^2\right)_n}\!-\!\frac{\left(\z q^2,q^3;q^2\right)_{\infty}}{\left(-q^3;q^2\right)_{\infty}}
\sum_{n\ge 1}\frac{\z^n q^{2n^2-n}}{\left(-\z q,-q^2;q^2\right)_n}
\right).
\end{multline}
Plugging in \eqref{Choi1} with $q\mapsto q^2, \zeta_1= \z q$, and $\zeta_2= q^2$ 
yields the theorem.
\end{proof}




\begin{proof}[Proof of \Cref{Thm1cor}]
Substituting $\z=q$ into the second identity of \Cref{Thm1}, using the definition of $f(\zeta_1,\zeta_2;q)$ from \eqref{Choi1} with $q\mapsto q^2$, $\zeta_1=\zeta_2=q^2$, and \eqref{Ramanujanf} with $q\mapsto q^2$, we get 
\begin{align*}
\overline{P}_{\textup{od}}(q;q)
&=\frac{\left(1+q\right)\left(-q;q^2\right)_{\infty}}{q^3\left( q^3;q^2\right)^2_{\infty}}\sum_{n\ge 1}\frac{\left(-q;q^2\right)_{n-1}q^{n^2+2n}}{\left(-q^2;q^2\right)^2_n}-\frac{1+q}{q^2}\left(f\left(q^2\right)-1\right).
\end{align*}
To complete the proof, we need to show that 
\begin{equation}\label{Thm1lem1eqn2}
\sum_{n\ge 1}\frac{\left(-q;q^2\right)_{n-1}q^{n^2+2n}}{\left(-q^2;q^2\right)^2_n}=\frac{q}{\left(1-q\right)^2}\left(-1-q+\mu(-q)\right).
\end{equation}
Applying \Cref{BEMeqn1} with $q\mapsto q^2, \z=c=-1$, and $d=q$, we have 
\[
\sum_{n\ge 1}\frac{\left(-\frac 1q;q^2\right)_n q^{n^2+2n}}{\left(-q^2;q^2\right)^2_n}
=-\frac{1+q}{1-q}\sum_{n\ge 1}\frac{\left(q;q^2\right)_{n}(-1)^n q^{2n}}{\left(-q^2;q^2\right)_n}.
\]
Therefore, we obtain
\begin{align*}
\sum_{n\ge 1}\frac{\left(-q;q^2\right)_{n-1}q^{n^2+2n}}{\left(-q^2;q^2\right)^2_n}
&=\frac{q}{\left(1-q\right)^2}\sum_{n\ge 1}\frac{\left(\frac 1q;q^2\right)_{n+1}\left(-1\right)^n q^{2n+1}}{\left(-q^2;q^2\right)_n}.
\end{align*}
Thus, following \eqref{McIntosh}, to prove \eqref{Thm1lem1eqn2}, it remains to show that
\begin{equation}\label{eqn1}
\sum_{n\ge 1}\frac{\left(\frac 1q;q^2\right)_{n+1}\left(-1\right)^n q^{2n+1}}{\left(-q^2;q^2\right)_n}=\sum_{n\ge 1}\frac{ \left(-q;q^2\right)_nq^{n^2}}{\left(-q^2;q^2\right)^2_n}-q.
\end{equation}
Using \Cref{BEMeqn1} with $q\mapsto q^2, \z=c=-1$, and $d=\frac 1q$, we have
\begin{equation*}
\sum_{n\ge 1}\frac{\left(-q;q^2\right)_n q^{n^2}}{\left(-q^2;q^2\right)^2_n}
=\frac{1+q}{1-q}\sum _{n\ge 1} \frac{\left(\frac 1q;q^2\right)_{n}(-1)^n q^{2n}}{\left(-q^2;q^2\right)_n}.
\end{equation*}
Thus, to prove \eqref{eqn1}, we need to show that
\begin{equation}\label{eqn1b}
\frac{1+q}{1-q}\sum _{n\ge 1} \frac{\left(\frac 1q;q^2\right)_{n}(-1)^n q^{2n}}{\left(-q^2;q^2\right)_n}-\sum_{n\ge 1}\frac{\left(\frac 1q;q^2\right)_{n+1}\left(-1\right)^n q^{2n+1}}{\left(-q^2;q^2\right)_n}=q.
\end{equation}

We prove this by the telescoping method. For $n\in \mathbb{N}$, define
\begin{equation}\label{newdef1}
\a_n:=\a_n(q):=\frac{\left(\frac 1q;q^2\right)_n(-1)^n q^{2n}}{\left(1-q\right)\left(-q^2;q^2\right)_{n-1}}.
\end{equation}
Note that
\begin{align*}
\a_n-\a_{n+1}
&=\frac{\left(1+q\right)\left(\frac 1q;q^2\right)_n(-1)^n q^{2n}}{\left(1-q\right)\left(-q^2;q^2\right)_n}-\frac{\left(\frac 1q;q^2\right)_{n+1}\left(-1\right)^n q^{2n+1}}{\left(-q^2;q^2\right)_n}.
\end{align*}
Consequently, we have
\begin{equation*}
\!\!\!\frac{1\!+\!q}{1\!-\!q}\sum _{n=1}^{N} \frac{\left(\frac 1q;q^2\right)_{\!n}(-1)^n q^{2n}}{\left(-q^2;q^2\right)_n}\!-\!\sum_{n=1}^{N}\frac{\left(\frac 1q;q^2\right)_{\!n+1}\left(-1\right)^n q^{2n+1}}{\left(-q^2;q^2\right)_n}=\sum_{n=1}^{N}(\a_n\!-\!\a_{n+1})=\a_1\!-\!\a_{N+1}.\!
\end{equation*}
From \eqref{newdef1}, we see that
\[
\a_n=\frac{\left(\frac 1q;q^2\right)_n(-1)^n q^{2n}}{\left(1-q\right)\left(-q^2;q^2\right)_{n-1}}=\frac{(-1)^{n+1}\left(q;q^2\right)_{n-1}q^{2n-1}}{\left(-q^2;q^2\right)_{n-1}}.
\]
Therefore, it follows that $\a_1=q$ and $\a_N\to 0$ as $N\to \infty$. This 
concludes the proof of \eqref{eqn1b}. 
\end{proof}

To prove \Cref{newthm1}, we use the following lemma.
\begin{lemma}\label{claim}
We have
\begin{equation*}
\sum_{n\ge1} \frac{q^{(2n-1)(n-1)}}{\left(1+q^{2n}\right)\left(-q\right)_{2n-2}}= 1-q \sum_{n\ge1} \frac{q^{2n^2-n}}{\left(-q\right)_{2n}}.
\end{equation*}
\end{lemma}
\begin{proof}
Making the change of variable $n\mapsto n+1$, we have 
\begin{align}\label{claimeqn1}
\sum_{n\ge1} \frac{q^{(2n-1)(n-1)}}{\left(1+q^{2n}\right)\left(-q\right)_{2n-2}}=\frac{1}{1+q^2}\sum_{n\ge0} \frac{q^{2n^2+n}}{\left(-q,-q^4;q^2\right)_{n}}.
\end{align}
Now, applying \Cref{Lostntbk} with $q\mapsto q^2, a= -q^2, b= -q^{-1}$ and using \Cref{RF} with $q\mapsto q^2,\alpha=0$, $\beta=-q^4$, and $t=-q$, we have 
\begin{align*}
\sum_{n\ge 0}\frac{q^{2n^2+n}}{\left(-q,-q^4;q^2\right)_n}
&
=1+q^2-q^2\sum_{n\ge0}\frac{q^{2n^2+3n}}{\left(-q^4;q^2\right)_n\left(-q;q^2\right)_{n+1}}.
\end{align*}
Applying this to \eqref{claimeqn1} and then letting $n\mapsto n-1$, the lemma follows. 
\end{proof}

%


\begin{proof}[Proof of Theorem \ref{newthm1}]
From \eqref{zgen2} and \eqref{Thm1eqn1prf}, we have 
\begin{align*}
&\sum_{n\ge 0}\frac{\left(-q^{2n+3};q^2\right)_{\infty}q^{2n+1}}{\left(\zeta q^{2n+2},q^{2n+3};q^2\right)_{\infty}}+\sum_{n\ge 1}\frac{\left(-q^{2n+1};q^2\right)_{\infty}q^{2n}}{\left(\zeta q^{2n},q^{2n+1};q^2\right)_{\infty}}\nonumber\\
&\hspace{2 cm}=\frac{\left(-q;q^2\right)_{\infty}}{\z q\left(\zeta q^2,q^3;q^2\right)_{\infty}}\Biggl(\sum_{n\ge 1}\frac{\left(-\frac 1q;q^2\right)_n\zeta^nq^{n(n+1)}}{\left(-\zeta q,-q^2;q^2\right)_n}
-\frac{\left(\zeta q^2,q^3;q^2\right)_{\infty}}{\left(-q^3;q^2\right)_{\infty}}\sum_{n\ge 1}\frac{\zeta^n q^{2n^2-n}}{\left(-\zeta q,-q^2;q^2\right)_n}\Biggr).	
\end{align*}
It is not hard to see that this is equivalent to
\begin{align*}
&\sum_{n\ge 1}\frac{\left(-q^{2n+3};q^2\right)_{\infty}q^{2n+1}}{\left(\zeta q^{2n+2},q^{2n+3};q^2\right)_{\infty}}+\sum_{n\ge 1}\frac{\left(-q^{2n+3};q^2\right)_{\infty}q^{2n+2}}{\left(\zeta q^{2n+2},q^{2n+3};q^2\right)_{\infty}}+\frac{1+q}{\z q}\sum_{n\ge 1}\frac{\zeta^n q^{2n^2-n}}{\left(-\zeta q,-q^2;q^2\right)_n}\\
&\hspace{2.5 cm}=\frac{\left(-q;q^2\right)_\infty}{\z q\left(\zeta q^2,q^3;q^2\right)_\infty} \sum_{n\ge1} \frac{\left(-\frac{1}{q};q^2\right)_n \zeta^n q^{n(n+1)}}{\left(-\zeta q,-q^2;q^2\right)_n} - \frac{\left(-q^3;q^2\right)_\infty q}{\left(\zeta q^2,q^3;q^2\right)_\infty} - \frac{\left(-q^3;q^2\right)_\infty q^2}{\left(\zeta q^2,q^3;q^2\right)_\infty}.
\end{align*}
Taking the limit  $\zeta\to q^{-2}$, we obtain
\begin{align}\label{special case Theorem 7.1}
&\sum_{n\ge1} \frac{\left(-q^{2n+3};q^2\right)_\infty q^{2n+1}}{\left(q^{2n},q^{2n+3};q^2\right)_\infty} + \sum_{n\ge1} \frac{\left(-q^{2n+3},q^2\right)_\infty q^{2n+2}}{\left(q^{2n},q^{2n+3};q^2\right)_\infty} + q(1+q)\sum_{n\ge1} \frac{q^{2n^2-3n}}{\left(-\frac{1}{q},-q^2;q^2\right)_n} \\
&\hspace{2.5 cm}=\lim_{\zeta\to q^{-2}} \frac{\dfrac{\left(-q;q^2\right)_\infty}{\z q\left(\zeta q^4,q^3;q^2\right)_\infty}\displaystyle\sum_{n\ge1} \dfrac{\left(-\frac{1}{q};q^2\right)_n \zeta^n q^{n(n+1)}}{\left(-\zeta q, -q^2;q^2\right)_n} - \dfrac{\left(-q;q^2\right)_\infty q}{\left(\zeta q^4,q^3;q^2\right)_\infty} 
}{1-\zeta q^2} =: L(q).\nonumber
\end{align}
We define 
\begin{align*}
N_r(\zeta;q)&:=\frac{\left(-q;q^2\right)_\infty}{\z q\left(\zeta q^4,q^3;q^2\right)_\infty} \sum_{n\ge1} \frac{\left(-\frac{1}{q};q^2\right)_n \zeta^n q^{n(n+1)}}{\left(-\zeta q, -q^2;q^2\right)_n} - \frac{\left(-q;q^2\right)_\infty q}{\left(\zeta q^4,q^3;q^2\right)_\infty},\ 
 N_r(q):= N_r\!\left(q^{-2};q\right).
\end{align*}
Using \Cref{kang}, we see that $N_r(q)=0$.
Therefore, applying L'Hospital's rule, it follows that 
\begin{align*}
\left[\frac{\partial}{\partial\zeta} N_r(\zeta;q)\right]_{\zeta=q^{-2}} 
\!\!\!\!\!\!\!\!\!\!\!=\frac{\left(-q;q^2\right)_\infty q}{(q^2)_\infty}\!\left(\!-q^2\!+\!q^2\sum_{n\ge1}\!\frac{nq^{n(n-1)}}{\left(-q^2;q^2\right)_n}\!-\!\sum_{n\ge1}\! \frac{q^{n(n-1)}}{\left(-q^2;q^2\right)_n}\sum_{k=1}^n\!\frac{q^{2k-1}}{1+q^{2k-3}}\right).
\end{align*}
Plugging this into \eqref{special case Theorem 7.1}, we obtain 
\begin{align*}
&\sum_{n\ge1}\frac{\left(-q^{2n+3};q^2\right)_\infty q^{2n+1}}{\left(q^{2n},q^{2n+3};q^2\right)_\infty}+\sum_{n\ge1} \frac{\left(-q^{2n+3};q^2\right)_\infty q^{2n+2}}{\left(q^{2n},q^{2n+3};q^2\right)_\infty}+q(1+q)\sum_{n\ge1} \frac{q^{2n^2-3n}}{\left(-\frac{1}{q},-q^2;q^2\right)_n}\nonumber\\
&\hspace{2.5 cm}=-\frac{\left(-q;q^2\right)_\infty}{q(q^2)_\infty}\left(-q^2+q^2\sum_{n\ge1} \frac{nq^{n(n-1)}}{\left(-q^2;q^2\right)_n} - \sum_{n\ge1} \frac{q^{n(n-1)}}{\left(-q^2;q^2\right)_n} \sum_{k=1}^n \frac{q^{2k-1}}{1+q^{2k-3}} \right).
\end{align*}
Simplifying the left-hand side and then dividing both sides by $q(1+q)$, we get
\begin{align*}
&\sum_{n\ge1}\frac{\left(-q^{2n+3};q^2\right)_\infty q^{2n}}{\left(q^{2n},q^{2n+3};q^2\right)_\infty}+\sum_{n\ge1} \frac{q^{2n^2-3n}}{\left(-\frac{1}{q},-q^2;q^2\right)_n}\\
&\hspace{2.5 cm}=\frac{\left(-q^3;q^2\right)_\infty}{(q^2)_\infty}\left(1-\sum_{n\ge1} \frac{nq^{n(n-1)}}{\left(-q^2;q^2\right)_n}+\sum_{n\ge1} \frac{q^{n(n-1)}}{\left(-q^2;q^2\right)_n} \sum_{k=1}^n \frac{q^{2k-3}}{1+q^{2k-3}} \right).
\end{align*}
Letting $k\mapsto k+1$ and then using \Cref{Donato} with $q\mapsto q^2$, we obtain 
\begin{align}\label{Verweis02}
\sum_{n\ge1}\frac{\left(-q^{2n+3};q^2\right)_\infty q^{2n}}{\left(q^{2n},q^{2n+3};q^2\right)_\infty}&=\frac{\left(-q^3;q^2\right)_\infty}{(q^2)_\infty}\left(1-\sigma\left(q^2\right)+\sum_{n\ge1} \frac{q^{n(n-1)}}{\left(-q^2;q^2\right)_n} \sum_{k=0}^{n-1} \frac{q^{2k-1}}{1+q^{2k-1}} \right)\nonumber\\
&\hspace{6 cm}-\sum_{n\ge1} \frac{q^{2n^2-3n}}{\left(-\frac{1}{q},-q^2;q^2\right)_n}.
\end{align}
We then simplify, using \Cref{claim}, 
\begin{align*}
\sum_{n\ge1} \frac{q^{2n^2-3n}}{\left(-\frac{1}{q},-q^2;q^2\right)_n}&
=\frac{1}{1+q}\sum_{n\ge1} \frac{q^{(2n-1)(n-1)}}{\left(1+q^{2n}\right)(-q)_{2n-2}}
=\frac{1}{1+q}\left(1-q\sum_{n\ge1}\frac{q^{2n^2-n}}{(-q)_{2n}}\right)\\
&=\frac{1}{1+q}\left(1-q\sum_{n\in\mathbb{Z}}\text{sgn}(n)q^{\frac{n(3n-1)}{2}}\right),
\end{align*}
where the last equality follows from \eqref{Neweqn2paper}. Plugging this into \eqref{Verweis02}, we obtain
\begin{align}\label{eqn0}
\sum_{n\ge1}\frac{\left(-q^{2n+3};q^2\right)_\infty q^{2n}}{\left(q^{2n},q^{2n+3};q^2\right)_\infty}&=\frac{\left(-q^3;q^2\right)_\infty}{(q^2)_\infty}\left(1-\sigma\left(q^2\right)+\sum_{n\ge1} \frac{q^{n(n-1)}}{\left(-q^2;q^2\right)_n} \sum_{k=0}^{n-1} \frac{q^{2k-1}}{1+q^{2k-1}} \right)\nonumber\\
&\hspace{4 cm}-\frac{1}{1+q}\left(1-q\sum_{n\in\mathbb{Z}}\text{sgn}(n)q^{\frac{n(3n-1)}{2}}\right).
\end{align}
Rearranging the double sum on the left-hand side of \eqref{eqn0}, we have
\begin{align}\label{eqn1n}
&\sum_{n\ge1}\frac{q^{n(n-1)}}{\left(-q^2;q^2\right)_n} \sum_{k=0}^{n-1} \frac{q^{2k-1}}{1+q^{2k-1}}
=\sum_{k\ge 0}\frac{q^{k^2+3k-1}}{\left(1+q^{2k-1}\right)\left(-q^2;q^2\right)_{k+1}}\lim_{t\to 0}\sum_{n\ge 0}\frac{\left(-\frac{q}{t};q^2\right)_n \left(tq^{2k+1}\right)^n}{\left(-q^{2k+4};q^2\right)_n}.
\end{align}
Applying \Cref{Heine} with $q\mapsto q^2$, $a= -\frac{q}{t}, b= q^2,\ \text{and}\ c= -q^{2k+4}$, we have 
\[
\lim_{t\to 0}\sum_{n\ge 0}\frac{\left(-\frac{q}{t};q^2\right)_n \left(tq^{2k+1}\right)^n}{\left(-q^{2k+4};q^2\right)_n}=\lim_{t\to0}\frac{\left(tq^{2k+3},-q^{2k+2};q^2\right)_\infty}{(-q^{2k+4},tq^{2k+1};q^2)_\infty}
=1+q^{2k+2}.
\]
Consequently, from \eqref{eqn1n}, it follows that 
\begin{align*}
&\sum_{n\ge1} \frac{q^{n(n-1)}}{\left(-q^2;q^2\right)_n} \sum_{k=0}^{n-1} \frac{q^{2k-1}}{1+q^{2k-1}}
=\sum_{k\ge 0}\frac{q^{k^2+3k-1}}{\left(1+q^{2k-1}\right)\left(-q^2;q^2\right)_k}
=\frac{1}{1+q}+\sum_{k\ge 0}\frac{q^{k^2+5k+3}}{\left(1+q^{2k+1}\right)\left(-q^2;q^2\right)_{k+1}}\\
&\hspace{3.5 cm}=\frac{1}{1+q}+\sum_{k\ge 1}\frac{q^{k(k+1)}}{\left(-q^2;q^2\right)_{k}}-\sum_{k\ge 0}\frac{q^{k^2+3k+2}}{\left(1+q^{2k+1}\right)\left(-q^2;q^2\right)_{k+1}}\nonumber\\
&\hspace{3.5 cm}=\frac{1}{1+q}+\sum_{k\ge 1}\frac{q^{k(k+1)}}{\left(-q^2;q^2\right)_{k}}-q\sum_{k\ge 0}\frac{q^{k(k+1)}}{\left(-q^2;q^2\right)_{k+1}}+\sum_{k\ge 0}\frac{q^{k(k+1)+1}}{\left(1+q^{2k+1}\right)\left(-q^2;q^2\right)_{k+1}}\nonumber\\
&\hspace{3.5 cm}=\frac{1}{1+q}+\sigma\left(q^2\right)-1-q+\sum_{k\ge 0}\frac{q^{k(k+1)+1}}{\left(1+q^{2k+1}\right)\left(-q^2;q^2\right)_{k+1}},
\end{align*}
using \Cref{kang} and \eqref{sigma function}. Thus, 
\begin{equation}\label{eqn4n}
1-\sigma\!\left(q^2\right)\! + \sum_{n\ge1} \frac{q^{n(n-1)}}{\left(-q^2;q^2\right)_n} \sum_{k=0}^{n-1} \frac{q^{2k-1}}{1+q^{2k-1}}=\frac{1}{1+q}-q+q\sum_{k\ge 0}\frac{q^{k(k+1)}}{\left(1+q^{2k+1}\right)\left(-q^2;q^2\right)_{k+1}}.
\end{equation}
Letting $k\mapsto k-1$, we have, using \Cref{kang}, \eqref{above3.6}, and \eqref{corson}, 
\begin{align*}
\sum_{k\ge 0}\frac{q^{k(k+1)}}{\left(1+q^{2k+1}\right)\left(-q^2;q^2\right)_{k+1}}
&=\sum_{k\ge 0}\frac{q^{k(k+1)}}{\left(-q^2;q^2\right)_{k+1}}-q^{-1}\sum_{k\ge 1}\frac{q^{k(k+1)}}{\left(1+q^{2k-1}\right)\left(-q^2;q^2\right)_{k}}\\
&=1-q^{-1}\sum_{k\ge 1}\frac{q^{k(k+1)}}{\left(1+q^{2k-1}\right)\left(-q^2;q^2\right)_{k}}
=\frac{1}{1-q}\left(W_1(q)-q\right).
\end{align*}
Plugging this into \eqref{eqn4n}, we obtain
\begin{align*}
1-\sigma\!\left(q^2\right)\!+ \sum_{n\ge1} \frac{q^{n(n-1)}}{\left(-q^2;q^2\right)_n} \sum_{k=0}^{n-1} \frac{q^{2k-1}}{1+q^{2k-1}}
=\frac{q}{1-q}W_1(q)+\frac{1-2q-q^2}{1-q^2}.
\end{align*}
Plugging this into \eqref{eqn0} and using \eqref{genfuncnew}, we conclude the proof.
\end{proof}


\begin{proof}[Proof of Theorem \ref{thm4}]
From \eqref{zgen3}, we have
\begin{align}\label{thm4eqn1}
\overline{P}_{\textup{ev}}(\z;q)
&=\frac{\z q\left(1+q\right)\left(-q^2,-\z q^3;q^2\right)_{\infty}}{\left(\z q^2;q^2\right)_{\infty}}\sum_{n\ge 0}\frac{\left(\z q^2;q^2\right)_n q^{2n}}{\left(-q^2,-\z q^3;q^2\right)_n}.
\end{align}
Applying \Cref{lemident2} with $q\mapsto q^2, a=1, b=\z q$, and $c=\z$, we have 
\begin{align}\label{thm4eqn2}
\sum_{n\ge 0}\frac{\left(\z q^2;q^2\right)_n q^{2n}}{\left(-q^2,-\z q^3;q^2\right)_n}
=\!-\frac{2\left(1+\z q\right)}{\z q}\!\sum_{n \ge 1}\!\frac{\left(\z q^2;q^2\right)_{n-1} (-1)^n\z^n q^{n^2}}{\left(-\z q,-\z q^2;q^2\right)_{n}}&\nonumber\\
&\hspace{-3.5 cm}\!+\!\frac{\left(\z q^2;q^2\right)_{\infty}}{\z q\left(-q^2, -\z q^3;q^2\right)_{\infty}}\!\sum_{n\ge 1}\frac{(-1)^n\z^{n}q^{n^2}}{\left(-\z q^2;q^2\right)_{n}}.
\end{align}
Applying \Cref{BEMeqn} with $q\mapsto q^2, b=-\z q, c=-\z$, and $d=-1$, it follows that 
\begin{align}\label{neweqn3}
\sum_{n\ge 1}\frac{\left(-\frac{\z q}{a}, \z;q^2\right)_n\!\left(-a\right)^n}{\left(-\z q, -\z q^2;q^2\right)_n}\!=\!\frac{\left(a\!+\!\z q\right)\left(-1\!+\!\z\right)}{-a\!+\!\z q}\!\sum_{n\ge 0}\frac{\left(a,-q;q^2\right)_n(-\z)^n}{\left(-\z q,-a;q^2\right)_n}\left(\frac{-a q^{2n}}{1\!+\!a q^{2n}}\!+\!\frac{\z q^{2n+1}}{1\!+\!\z q^{2n+1}}\right).
\end{align}
Letting $a\to 0$ on the left-hand side of \eqref{neweqn3}, we have, using $\lim_{a\to0} (-\frac{\zeta q}{a}; q^2)_n a^n = \zeta^n q^{n^2}$, 
\begin{equation}\label{neweqn1}
\lim\limits_{a\to 0}\sum_{n\ge 1}\frac{\left(-\frac{\z q}{a}, \z;q^2\right)_n \left(-a\right)^n}{\left(-\z q, -\z q^2;q^2\right)_n}
=\left(1-\z\right)\sum_{n\ge 1}\frac{\left(\z q^2;q^2\right)_{n-1}(-1)^n \z^n q^{n^2}}{\left(-\z q, -\z q^2;q^2\right)_n}.
\end{equation}
Letting $a\to 0$ on the right-hand side of \eqref{neweqn3} and substituting the resulting identity and \eqref{neweqn1} into \eqref{neweqn3} and then dividing by $1-\z$, we get
\[
\sum_{n \ge 1}\frac{\left(\z q^2;q^2\right)_{n-1} (-1)^n \z^n q^{n^2}}{\left(-\z q^2,-\z q;q^2\right)_{n}}=-\frac{\z q}{1+\z q}\sum_{n\ge 0}\frac{\left(-q;q^2\right)_n (-1)^n\z^nq^{2n}}{\left(-\z q^3;q^2\right)_n}.
\]
Plugging this into \eqref{thm4eqn2}, we then obtain
\begin{align*}
&\sum_{n\ge 0}\frac{\left(\z q^2;q^2\right)_n q^{2n}}{\left(-q^2,-\z q^3;q^2\right)_n}\!=\!2\sum_{n\ge 0}\frac{\left(-q;q^2\right)_n(-1)^n\z^nq^{2n}}{\left(-\z q^3;q^2\right)_n}+\frac{\left(\z q^2;q^2\right)_{\infty}}{\z q\left(-q^2, -\z q^3;q^2\right)_{\infty}}\!\sum_{n\ge 1}\frac{(-1)^n \z^nq^{n^2}}{\left(-\z q^2;q^2\right)_{n}}\\
&= \frac{\left(q,\z q^2\right)_{\infty}\left(\frac{1}{\z q};q^2\right)_{\infty}}{\left(-\z q^2\right)_{\infty}}+\!\left(1+\frac{1}{\z q}\right)\sum_{n\ge 0}\frac{\left(-\frac{1}{\z};q^2\right)_n(-1)^n q^{n}}{\left(-q^2;q^2\right)_n}+\frac{\left(\z q^2;q^2\right)_{\infty}}{\z q\left(-q^2, -\z q^3;q^2\right)_{\infty}}\!\sum_{n\ge 1}\frac{(-1)^n \z^nq^{n^2}}{\left(-\z q^2;q^2\right)_{n}}.
\end{align*}
using \Cref{Agarwal}. 
Consequently, by \eqref{thm4eqn1}, we get 
\begin{align}\label{thm4neweqn1}
	\overline{P}_{\textup{ev}}(\z;q)=\frac{\z q\left(q^2,\z q^3,\frac{1}{\z q};q^2\right)_{\infty}}{\left(-q^3,-\z q^2;q^2\right)_{\infty}}&+\frac{\left(1+q\right)\left(-q^2,-\z q;q^2\right)_{\infty}}{\left(\z q^2;q^2\right)_{\infty}}\sum_{n\ge 0}\frac{\left(-\frac{1}{\z};q^2\right)_n(-1)^n q^{n}}{\left(-q^2;q^2\right)_n}\nonumber\\
&\hspace{2 cm}+\left(1+q\right)\sum_{n\ge 1}\frac{(-1)^n \z^n q^{n^2}}{\left(-\z q^2;q^2\right)_n}.
\end{align}
Using the definition of $\nu$ 
from \eqref{Choi2}, we conclude the proof.  
\end{proof} 



\begin{proof}[Proof of \Cref{cor7.4}]
	
Plugging $\z=-q$ into 
	\Cref{thm4}, \eqref{Choi2}, and \eqref{ramanujannu}, we have 
	\begin{align}\label{cor7.4eqn1}
	\overline{P}_{\textup{ev}}(-q;q)&=-(1+q)-\frac{2\left(-q^2;q^4\right)_{\infty}\left(q^8;q^8\right)_{\infty}}{(q^6;q^4)_{\infty}}+\frac{\left(1+q\right)\left(q^4;q^4\right)_{\infty}}{(-q^3;q^2)_{\infty}}\sum_{n\geq0}\frac{\left(\frac 1q;q^2\right)_n(-1)^n q^n}{(-q^2;q^2)_n}\nonumber\\
	&\hspace{8.5 cm}+\left(1-q^2\right)\nu(-q).
	\end{align}	
	Using \Cref{RF} with $q\mapsto q^2, \a= \frac{1}{q}, \b=-q^2$, and $t=-q$, we obtain 
	\[
	\sum_{n\geq0}\frac{\left(\frac 1q;q^2\right)_n\!(-1)^nq^n}{(-q^2;q^2)_n}=\frac{2}{1+q},
	\]
employing the fact $(1;q^2)_n=0$ for $n\in \mathbb{N}$.	Plugging this into \eqref{cor7.4eqn1}, we conclude the proof.	
\end{proof}	



\begin{proof}[Proof of \Cref{newthm2}]
Using \eqref{Choi2}, we obtain 
\[
\left(1+\z\right)\nu\left(iq,i\left(\z q\right)^{\frac 12};q\right)
=\sum_{n\ge 0}\frac{(-1)^n \z^n q^{n^2}}{\left(-\z q^2;q^2\right)_{n}}.
\]	
Plugging this into \Cref{thm4} with $\z=-1$, we obtain 
\begin{align*}
\overline{P}_{\textup{ev}}(-1;q)
&=-\left(1+q\right)\left(-q;q^2\right)_{\infty}+(1+q)\left(q;q^2\right)_{\infty}+(1+q)\left(\sum_{n\ge 0}\frac{ q^{n^2}}{\left( q^2;q^2\right)_{n}}-1\right).
\end{align*}
Applying \Cref{Euler} 
 to this, 
  we conclude the proof. 
\end{proof}



\begin{proof}[Proof of \Cref{RamanujanIdentity}]
Using \eqref{zgen3}, splitting off the terms $n=1$ on both of the sums on the left-hand side of \eqref{thm4neweqn1}, rearranging, and letting $\z\to q^{-2}$, we obtain	
\begin{align}\label{R2}
&\sum_{n\ge 1}\frac{\left(-q^{2n+2},-q^{2n+1};q^2\right)_{\infty} q^{2n-1}}{\left( q^{2n};q^2\right)_{\infty}}+\sum_{n\ge 1}\frac{\left(-q^{2n+2},- q^{2n+1};q^2\right)_{\infty} q^{2n}}{\left( q^{2n};q^2\right)_{\infty}}-\frac{\left(q^2,q,q;q^2\right)_{\infty}}{q\left(-q^3,-1;q^2\right)_{\infty}}\nonumber\\
&\hspace{7 cm}-\left(1+q\right)\sum_{n\ge 1}\frac{(-1)^n q^{n^2-2n}}{\left(-1;q^2\right)_{n}}=\frac{\mathcal{L}(q)}{\left(q^2;q^2\right)_{\infty}},
\end{align}
where
\begin{equation}\label{L}
\mathcal{L}(q)\!:=\!\lim_{\z\to q^{-2}}\!\frac{\left(1+q\right)\left(-q^2,-\z q;q^2\right)_{\infty}\sum_{n\ge 0}\frac{\left(-\frac{1}{\z};q^2\right)_n(-1)^n q^{n}}{\left(-q^2;q^2\right)_n}-\left(-q^{2},-\z q^{3};q^2\right)_{\infty}\z q\left(1+q\right)}{1-\z q^2}.
\end{equation}
Applying L'Hospital, \Cref{AGLeqn}, \eqref{Ramtheta}, and \eqref{L}, we obtain
\[
\frac{\mathcal{L}(q)}{\left(q^2;q^2\right)_{\infty}}=
\frac{\left(1\!+\!q\right)\left(1-\Theta^2(-q)\right)}{4q\left(q\right)_{\infty}}.
\]
 Plugging this into \eqref{R2}, using \eqref{Ramanujanphi}, using \eqref{Ramtheta}, \cite[equation (1.6)]{andrews-kumar}, \cite[equation (26.27)]{Fine}, and \eqref{Ramanujanf}, we conclude the proof.\qedhere  
\end{proof}



\begin{proof}[Proof of \Cref{thm3}] By \Cref{lemident1} with $q\mapsto q^2$, $a=-\z q^{-1}, b=-\z^{-1}q^{-1}$, and $c=q^{-1}$,  
\begin{equation*}
\sum_{n\ge 0}\!\frac{\left(\z q,\z^{-1}q;q^2\right)_nq^{2n+2}}{\left(-q;q^2\right)_n}\!=\!\sum_{n\ge 1}\!\frac{\left(-q;q^2\right)_n\!q^{n^2+1}}{\left(-\z q^2,-\z^{-1}q^2;q^2\right)_n}-\frac{\left(\z q,\z^{-1}q;q^2\right)_{\infty}}{\left(-q;q^2\right)_{\infty}}\!\sum_{n\ge 1}\!\frac{q^{2n^2+1}}{\left(-\z q^2,-\z^{-1}q^2;q^2\right)_n}.
\end{equation*}
Applying \eqref{rank1} and \eqref{rank2}, we get 
\begin{align*}
\sum_{n\ge 0}\!\frac{\left(\z q,\z^{-1}q;q^2\right)_nq^{2n}}{\left(-q;q^2\right)_n}
&=q^{-1}\left(\operatorname{R2}\left(-\z;q\right)-1-\frac{\left(\z q,\z^{-1}q;q^2\right)_{\infty}}{\left(-q;q^2\right)_{\infty}}\left(R\left(-\z;q^2\right)-1\right)\right).
\end{align*}
Using this and \eqref{zgen1}, we then conclude the proof. 
\end{proof}

\section{Concluding Remarks}\label{cr}
  The results of this paper suggest several directions for further investigation. In particular, it would be interesting to understand more systematically when restrictions on overpartitions lead to generating functions exhibiting mixed modular structures. Another natural problem is to study whether similar phenomena occur for other partition statistics or colored partition families. The results of this paper suggest several directions for further investigation. 
\begin{enumerate}
	\item In Theorems \ref{Thm1} and \ref{thm4}, for special choices of $\z$, the functions $P_{\textup{od}}(\z;q)$ and $P_{\textup{ev}}(\z;q)$ are connected to various modular objects. Do $f(\z q, q^2;q^2)$ and $\nu(iq,i(\z q)^{\frac 12};q)$ fit into the theory of mock modular forms if $\z$ is a root of unity? In \Cref{Thm1} and \Cref{thm4}, for specific choices of $\z$, we study the modularity of 
	\[
	F_1(\z;q):=\sum_{n\ge 1}\!\frac{\left(-\frac 1q;q^2\right)_n\!\z^nq^{n(n+1)}}{\left(-\z q,-q^2;q^2\right)_n}\ \ \text{and}\ \ F_2(\z;q):=\sum_{n\ge 0}\frac{\left(-\frac{1}{\z};q^2\right)_n(-1)^n q^{n}}{\left(-q^2;q^2\right)_n}.
	\]
 Are $F_1(\z;q)$ or $F_2(\z;q)$ modular for other choices?
 \item Give a bijective proof of \eqref{reln}.
 \item  The generating functions studied here exhibit combinations of several different modular objects, including modular forms, mock theta functions, mock Maass theta functions, and false theta functions. It would be interesting to understand more systematically when restrictions on overpartitions lead to generating functions with such mixed modular structures.
 \item Another direction would be to define and study additional rank or crank-type statistics for the restricted overpartitions considered here and to investigate the modular properties of their generating functions.
 \item In this paper, we determine the asymptotic main terms of several counting functions. It would be interesting to obtain more precise asymptotic expansions, possibly using the Circle Method or modular transformations.
 \item It would be interesting to see whether the restricted overpartition functions studied here satisfy arithmetic congruences analogous to those known for the classical partition function.
\end{enumerate}

\end{document}